\documentclass[11pt]{article}
\usepackage{amsfonts,amsmath,amsthm,amssymb,enumerate,color}
\usepackage[makeroom]{cancel}
\usepackage{authblk}

\numberwithin{equation}{section}
\newcommand{\beq}{\begin{equation}}
\newcommand{\eeq}{\end{equation}}
\newcommand{\bea}{\begin{eqnarray}}
\newcommand{\eea}{\end{eqnarray}}
\newcommand{\beas}{\begin{eqnarray*}}
\newcommand{\eeas}{\end{eqnarray*}}
\newcommand{\ds}{\displaystyle}
\newtheorem{theorem}{Theorem}[section]

\newtheorem{definition}[theorem]{Definition}
\newtheorem{proposition}[theorem]{Proposition}

\newtheorem{corollary}[theorem]{Corollary}
\newtheorem{lemma}[theorem]{Lemma}
\newtheorem{remark}[theorem]{Remark}
\newtheorem{example}[theorem]{Example}
\newtheorem{examples}[theorem]{Examples}
\newtheorem{foo}[theorem]{Remarks}

\newcommand{\ee}{\ell}

\newcommand{\bM}{\mathbb M}

\newcommand{\Ho}{\mathcal H}
\newcommand{\V}{\mathcal V}

\newcommand{\M}{\mathbb M}

\newcommand{\R}{\mathbb R}

\newcommand{\ve}{\varepsilon}

\newcommand{\ch}{\mathcal H}

\newcommand{\ep}{\varepsilon}
\newcommand{\Dh}{{\Delta}_{\mathcal{H}}}
\newcommand{\Dv}{{\Delta}_{\mathcal{V}}}

\title{Sub-Laplacian comparison theorems on totally geodesic Riemannian foliations}

\author[1]{Fabrice Baudoin
}
\author[2]{Erlend Grong}
\author[3]{Kazumasa Kuwada}
\author[4]{Anton Thalmaier}
\affil[1]{Department of Mathematics, University of Connecticut,\par
   341 Mansfield Road, Storrs, CT  06269-1009, USA\par
   \texttt{fabrice.baudoin@uconn.edu}\vspace{1em}}

\affil[2]{Universit\'e Paris-Sud, LSS-SUP\'ELEC, \par
3, rue Joliot-Curie, 91192 Gif-sur-Yvette, France, and \par
   University of Bergen, Department of Mathematics, \par P.O. Box 7803, 5020 Bergen, Norway\par
   \texttt{erlend.grong@math.uib.no}\vspace{1em}}

\affil[3]{Mathematical Institute, Graduate School of Science,\par
   Tohoku University, 980-8578, Sendai, Japan\par
   \texttt{kuwada@m.tohoku.ac.jp}\vspace{1em}}

\affil[4]{Mathematics Research Unit, University of Luxembourg,\par 
   L-4364 Esch-sur-Alzette, Luxembourg\par 
   \texttt{anton.thalmaier@uni.lu}}

\begin{document}

\maketitle

\newcommand\blfootnote[1]{%
  \begingroup
  \renewcommand\thefootnote{}\footnote{#1}%
  \addtocounter{footnote}{-1}%
  \endgroup
}
\blfootnote{First author supported in part by NSF Grant DMS 1660031.}
\blfootnote{The second author supported by project 249980/F20 of the Norwegian Research Council.}
\blfootnote{The third author supported by JSPS Grant-in-Aid for Young Scientist (KAKENHI) 26707004.}
\blfootnote{Fourth author supported by FNR Luxembourg: OPEN scheme
(project GEOMREV O14/7628746).}

\begin{abstract}
We develop a variational theory of geodesics for the canonical variation of the metric of a totally geodesic foliation. As a consequence, we obtain  comparison theorems for the horizontal and vertical Laplacians. In the case of Sasakian foliations, we show that sharp horizontal and vertical Laplacian  comparison theorems for the sub-Riemannian  distance may  be obtained as a limit of horizontal and vertical Laplacian comparison theorems for the Riemannian distances  approximations. As a corollary we prove that, under suitable curvature conditions,  sub-Riemannian Sasakian spaces are actually limits  of Riemannian spaces satisfying  a uniform measure contraction property.
\end{abstract}

\

\tableofcontents

\section{Introduction}

In the last few years there has been major progress in understanding curvature type invariants in sub-Riemannian geometry and their applications to partial differential equations. In that topic, one can distinguish two main lines of research:
\begin{itemize}
\item A Lagrangian approach to curvature which is based on second variation formulas for sub-Riemannian geodesics and an intrinsic theory of sub-Riemannian Jacobi fields. We refer to \cite{AZ1,AZ2} and to the recent memoir \cite{ABR2} and its bibliography for this theory.
\item An Eulerian approach to curvature which is based on Bochner type inequalities for the sub-Laplacian as initiated in \cite{BG} (see also \cite{GT1,GT2}).
\end{itemize}
The two methods have their own advantages and inconveniences.
The first approach is more intrinsic and yields curvature invariants from the sub-Riemannian structure only.  Though it gives a deep understanding of the geodesics and, in principle, provides a  general framework, it is somehow challenging to compute and to make use of those invariants, even in simple examples like Sasakian spaces (see \cite{ABR,LL,LLZ}). The second approach is more extrinsic and produces curvature quantities from the sub-Riemannian structure together with the choice of a \textit{natural} complement to the horizontal distribution. Actually, the main idea in \cite{BG} is to embed the sub-Riemannian structure into a family or Riemannian structures converging  to the sub-Riemannian one. Sub-Riemannian curvature invariants appear then as the tensors controlling, in a certain sense, this convergence. Since it requires the existence of a good complement allowing the embedding, this approach is a priori  less general but it has the advantage to make available the full power of Riemannian tensorial methods to large classes of sub-Riemannian structures and is more suited to the study of subelliptic PDEs and their connections to the geometry of the ambient space (see for instance \cite{Bau} for a survey).

\

In the present paper, we aim at filling a gap between those two approaches by studying the variational theory of the geodesics of the Riemannian approximations in the setting of totally geodesic foliations.  Our framework is the following. Let $(\M,g,\mathcal{F})$ be a totally geodesic Riemannian foliation on a manifold $\M$ with horizontal bracket generating distribution $\Ho$. The sub-Riemannian structure we are interested in is $(\M, \Ho, g_{\Ho})$ where $g_\Ho$ is the restriction of $g$ to $\Ho$. It can be approximated by the family of Riemannian manifolds $(\M,g_\ve)$ obtained by blowing up the metric $g$ in the direction of the leaves (see formula \ref{cv}). A natural sub-Laplacian for  $(\M, \Ho, g_{\Ho})$ is the horizontal Laplacian $\Delta_\Ho$ of the foliation. Our main interest is then in uniform  Hessian and sub-Laplacian comparison theorems for the Riemannian distances approximations of the sub-Riemannian distance. Namely, we wish to estimate $\Delta_\Ho r_\ve$ everywhere it is defined, where $r_\ve$ denotes the distance from a fixed point for the distance associated to $g_\ve$ and deduce a possible limit comparison theorem for $\Delta_\Ho r_0$, where $r_0$ denotes now the sub-Riemannian distance.  Obviously, relevant estimates may not be obtained by standard Riemannian comparison geometry based on Ricci curvature. Indeed, the basic idea in classical comparison theory is to compare the geometry of the manifold to the geometry of model spaces which are isotropic in the sense that all directions are the same for the energy cost of geodesics. In our setting, when $\ve \to 0$ the horizontal directions are preferred and geodesics actually do converge to horizontal curves. Quantitatively, when $\ve \to 0$ the Riemannian Ricci curvature of the metric $g_\ve$  diverges to $-\infty$ in the horizontal directions and $+\infty$ in the vertical directions. To obtain relevant uniform estimates for $\Delta_\Ho r_\ve$, it is therefore more natural to develop a comparison geometry with respect to \textit{model foliations}. In all generality, the classification of such model foliations is a difficult task. However, when the foliation is of Sasakian type  it becomes possible to develop a sectional curvature comparison theory with respect to the models:
\begin{itemize}
\item The Heisenberg group as a flat model;
\item The Hopf fibration $\mathbb{S}^1 \to \mathbb{S}^{2n+1} \to \mathbb{CP}^n$ as a positively curved model;
\item The universal cover of the anti de-Sitter fibration $\mathbb{S}^1 \to \mathbf{AdS}_{2n+1} \to \mathbb{CH}^n$ as a negatively curved model.
\end{itemize}
 This point of view will allow us to prove a horizontal Hessian comparison theorem, as well as a  uniform sub-Laplacian comparison theorem for $\Delta_\Ho r_\ve$ that actually has a limit when $\ve \to 0$ (see Theorem \ref{th:SasakianComp3}). For instance, we obtain that for non-negatively curved Sasakian foliations (in the sense of Theorem \ref{th:SasakianComp3}), one has:
\[
\Delta_\Ho r_0 \le \frac{n+2}{r_0}
\]
where $n$ is the dimension of the horizontal distribution. In view of the known results by Agrachev \& Lee in dimension 3 (see \cite{AL1,AL2}), the constant $n+2$ is sharp.

\

The paper is organized as follows. In Section~\ref{sec:RFoliation}, we work in any totally geodesic foliations and compute the second variation formula of Riemannian $g_\ve$-geodesics with respect to variations in horizontal directions only. As a consequence we deduce a first  family of sub-Laplacian comparison theorems under weak and general conditions (see Theorem~\ref{comparison 1}). We deduce several consequences of those estimates, like a sharp injectivity radius estimate (Corollary \ref{injectivity}) or a Bonnet-Myers type theorem (Corollary \ref{BMyers}). It is remarkable, but maybe unsurprising, that the tensors controlling the trace of the horizontal index form are the same tensors that appear in the Weitzenb\"ock formula (see \cite{BKW,GT3}) for the  sub-Laplacian. In fact, we will prove in Section~\ref{ssec:HorVerComp} that this family of sub-Laplacians comparison theorems may actually also be proved by using the generalized curvature dimension inequalities introduced in \cite{BG,GT1,GT2}. Though the generalized curvature dimension inequality implies many expected byproducts of a sub-Laplacian comparison theorem like uniform volume doubling properties for the sub-Riemannian balls (see \cite{BBG}), there is no limit in Theorem \ref{comparison 1} when $\ve \to 0$. It seems that stronger geometric conditions are needed to prove a uniform family of sub-Laplacian comparison theorems that has a limit when $\ve \to 0$. To the best of our knowledge, it is therefore still an open question to decide whether the sub-Riemannian curvature dimension inequality alone implies or not a measure contraction property of the underlying metric measure space.

\

In Section~\ref{sec:HTypeComp}, we work in the framework of Sasakian foliations  and prove under suitable conditions a uniform family of horizontal Hessian and sub-Laplacian comparison theorems. It should come as no surprise that for the sub-Laplacian comparison theorem, the assumptions are stronger than in Section~\ref{sec:RFoliation}. The main theorem is Theorem~\ref{th:SasakianComp3}. It is proved as a consequence of a uniform family of Hessian comparison theorems (Theorem \ref{Hessian general3}). The idea behind the proof of Theorem \ref{Hessian general3} is pointed out above: we develop a comparison geometry with respect to Sasakian model spaces of constant curvature. In those Sasakian model spaces Jacobi fields can be computed explicitly (see Appendix 2). We point out that the computation of Jacobi fields in those model spaces is not straightforward, and to the best of our knowledge is new in this form. The novelty in our computations is that we work with a family of connections first introduced in \cite{Bau}. These connections are natural generalizations of the Levi-Civita connection and  are suited to the setting of Riemannian foliations with totally geodesic leaves. Though the connections are not torsion free, their adjoints are metric, and it is therefore easy to develop the formalism of Jacobi fields in this framework (see Appendix 1). In the final part of the paper, we explore then some consequences of the sub-Laplacian comparison theorems in terms of measure contraction properties. In particular, in the non-negatively curved case we obtain the interesting fact that the family of Riemannian manifolds $(\M,g_\ve)$, $\ve >0$, uniformly satisfies the measure contraction properties $\mathbf{MCP}(0,n+4)$ despite the fact that when $\ve \to 0$ the Riemannian Ricci curvature of the metric $g_\ve$  diverges to $-\infty$ in the horizontal directions and $+\infty$ in the vertical directions. We also obtain sharp sub-Riemannian type Bonnet-Myers theorems (see Theorem \ref{BMsubriem}).

\

\textbf{Acknowledgments:} \textit{The first author would like to thank Nicola Garofalo for stimulating discussions on methods related to Section~\ref{ssec:HorVerComp}.}

\section{Horizontal and vertical Laplacian comparison theorems on Riemannian foliations} \label{sec:RFoliation}

\subsection{Framework} \label{ssec:foliation}

Throughout the paper, we consider a smooth connected $n+m$ dimensional manifold $\M$ which is equipped with a Riemannian foliation with a bundle like  metric $g$ and totally geodesic $m$ dimensional leaves. We moreover always  assume that the metric $g$ is complete and that the horizontal distribution $\mathcal{H}$ of the foliation is bracket-generating.  We denote by $\mu$ the Riemannian reference volume  measure on $\M$.

\

As is  usual, the sub-bundle $\mathcal{V}$ formed by vectors tangent to the leaves is referred  to as the set of \emph{vertical directions}. The sub-bundle $\mathcal{H}$ which is normal to $\mathcal{V}$ is referred to as the set of \emph{horizontal directions}.   Saying that the foliation is totally geodesic and Riemannian means that:
\begin{equation} \label{RTG} (\mathcal{L}_X g)(Z,Z)=0, \qquad (\mathcal{L}_Z g)(X, X)=0, \qquad \text{for any $X \in \Gamma^\infty(\Ho)$, $Z \in \Gamma^\infty(\V)$.}\end{equation}
The literature on Riemannian foliations is vast, we refer for instance  to the classical reference  \cite{Tondeur} and its bibliography for further  details.

\

The Riemannian gradient will be denoted $\nabla$ and we write the horizontal gradient as $\nabla_\Ho$, which is the projection of $\nabla$ onto $\mathcal{H}$. Likewise, $\nabla_\mathcal{V}$ will denote the vertical gradient.  The horizontal Laplacian $\Delta_\Ho$ is the generator of the symmetric closable bilinear form:
\[
\mathcal{E}_{\mathcal{H}} (f,g) 
=-\int_\bM \langle \nabla_\mathcal{H} f , \nabla_\mathcal{H} g \rangle_{\mathcal{H}} \,d\mu, 
\quad f,g \in C_0^\infty(\M).
\]
The vertical Laplacian may be defined as $\Delta_\V=\Delta -\Delta_\Ho$ where $\Delta$ is the Laplace-Beltrami operator on $\M$. We have
\[
\mathcal{E}_{\mathcal{V}} (f,g) :=-\int_\bM \langle \nabla_\mathcal{V} f , \nabla_\mathcal{V} g \rangle_{\mathcal{V}} \,d\mu=\int_\M f \Delta_\V g \,d\mu, \quad f,g \in C_0^\infty(\M).
\]

The hypothesis that $\mathcal{H}$ is bracket generating implies that the horizontal Laplacian $\Delta_{\mathcal{H}}$ is locally subelliptic and the completeness assumption on $g$ implies that $\Delta_{\mathcal{H}}$ is essentially self-adjoint on the space of smooth and compactly supported functions (see for instance~\cite{B2}).

\subsection{Canonical variation of the metric}

In this section, we introduce the canonical variation of the metric and study some of its basic properties. The Riemannian metric $g$ can be split as
\begin{equation} \label{eq:g-split}
g=g_\mathcal{H} \oplus g_{\mathcal{V}},
\end{equation}
and  we introduce the one-parameter family of rescaled Riemannian metrics:
\begin{align}\label{cv}
g_{\varepsilon}=g_\mathcal{H} \oplus  \frac{1}{\varepsilon }g_{\mathcal{V}}, \quad \varepsilon >0.
\end{align}
It is called the canonical variation of $g$ (see \cite{Besse}, Chapter 9, for a discussion in the submersion case). The Riemannian distance associated with $g_{\varepsilon}$ will be denoted by $d_\varepsilon$. It should be noted that $d_\varepsilon$, $\varepsilon >0$, form an increasing (as $\epsilon\downarrow 0$) family of distances converging pointwise to the sub-Riemannian distance $d_0$.

\

Let $x_0\in \M$ be fixed and for $\varepsilon \ge 0$ denote
\[
r_\varepsilon (x) =d_\varepsilon (x_0,x).
\]
The cut-locus $\mathbf{Cut}_\ve (x_0)$ of  $x_0$ for the distance $d_\ve$ is defined as the complement of the set of $y$'s in $\M$ such that there exists a unique length minimizing normal geodesic joining $x_0$ and $y$ and $x_0$ and $y$ are not conjugate along such geodesic (see \cite{A}). The global cut-locus of $\M$ is defined by
\[
\mathbf{Cut}_\ve (\M)=\left\{ (x,y) \in \M \times \M,\  y \in \mathbf{Cut}_\ve (x) \right\}.
\]

\begin{lemma}[\cite{A}, \cite{RT}]\label{cutlocus}
Let $\ve \ge 0$. The following statements hold:
 \begin{enumerate}[\rm 1.]
\item The set $\M \setminus \mathbf{Cut}_\ve (x_0)$ is open and dense in $\M$.
\item The function $(x,y) \to d_\ve  (x,y)^2$ is smooth on $\M \times \M \setminus \mathbf{Cut}_\ve (\M)$.
\end{enumerate}
\end{lemma}

It is proved in \cite{B2} that since the foliation is totally geodesic, we have for every $\varepsilon >0$,
\[
\Gamma ( f , \| \nabla^{g_\varepsilon} f \|_{g_\varepsilon}^2)=\langle \nabla^{g_\varepsilon} f , \nabla^{g_\varepsilon} \Gamma(f) \rangle_{g_\varepsilon}
\]
where $\Gamma(f)=\| \nabla_\mathcal{H} f \|_g^2$ is the \textit{carr\'e du champ operator} of $ \Delta_{\mathcal{H}}$ and $\nabla^{g_\varepsilon}$ the Riemannian gradient for the metric $g_\varepsilon$.
Applying this equality with $f=r_\ve$, we obtain that outside of the cut-locus of $x_0$,
\begin{align}\label{observation}
\langle \nabla^{g_\varepsilon} r_\ve , \nabla^{g_\varepsilon} \Gamma(r_\ve) \rangle_{g_\varepsilon}=0.
\end{align}
This implies that $\Gamma (r_\varepsilon)$ is constant on $g_\varepsilon$ distance minimizing geodesics issued from $x_0$. Likewise, denoting $\Gamma^\V(f)=\| \nabla_\V f \|_g^2$, we obtain that $\Gamma^\V (r_\ve)$ is constant on $g_\varepsilon$ distance minimizing geodesics issued from $x_0$.

\

The following lemma will be useful:

\begin{lemma}\label{limit}
Let $x  \in \M$, $x \neq x_0$ which is not in $\cup_{n \ge 1}  \mathbf{Cut}_{1/n} (x_0)$, then
\[
\lim_{n \to +\infty } \| \nabla_\Ho r_{1/n} (x) \|_g =1.
\]
\end{lemma}

\begin{proof}
Let $\gamma_n :[0,1]\to \M$ be the unique, constant speed, and length minimizing $g_{1/n}$ geodesic connecting $x_0$ to $x$. From \eqref{observation}, one has $  d_{1/n}(x_0,x)\| \nabla_\Ho r_{1/n} (x) \|_g=\| \gamma'_n(0) \|_\Ho$. We therefore need to prove that $\lim_{n \to \infty} \| \gamma'_n(0) \|_\Ho=d_0(x_0,x)$. Let us observe that
\[
\| \gamma'_n(0) \|^2_\Ho +n \| \gamma'_n(0) \|^2_\V = d_{1/n}(x_0,x)^2.
\]
Therefore, $\lim_{n \to \infty} \| \gamma'_n(0) \|^2_\V=0$. Let us now assume that $\| \gamma'_n(0) \|_\Ho$ does not converge to $d_0(x_0,x)$. In that case, there exists a subsequence $n_k$ such that $\| \gamma'_{n_k}(0) \|_\Ho$ converges to some $0 \le a <d_0(x_0,x)$. For $f \in C_0^\infty (\M)$ and $0 \le s \le t \le 1$, we have
\[
| f(\gamma_{n_k}(t))-f (\gamma_{n_k}(s))| \le \left(\| \gamma'_{n_k}(0) \|_\Ho \| \nabla_\Ho f \|_\infty + \| \gamma'_{n_k}(0) \|_\V \| \nabla_\V f \|_\infty \right)   (t-s).
\]
From Arzel\`a-Ascoli's theorem we deduce that there exists a subsequence which we continue to denote $\gamma_{n_k}$ that converges uniformly to an absolutely continuous curve $\gamma$, such that $\gamma(0)=x_0$, $\gamma(1)=x$. We have  for $f \in C_0^\infty (\M)$ and $0 \le s \le t \le 1$, 
\[
| f(\gamma (t))-f (\gamma (s))| \le a  \| \nabla_\Ho f \|_\infty (t-s).
\]
In particular, we deduce that
\[
| f(x)-f (x_0)| \le a  \| \nabla_\Ho f \|_\infty .
\]
Since it holds for every $f \in C_0^\infty (\M)$, one deduces 
\[
d_0(x_0,x)=\sup \{ |f(x)-f(x_0)|,\ f \in C_0^\infty (\M),\ \| \nabla_\Ho f \|_\infty \le 1\}  \le a .
\] 
This contradicts the fact that $a< d_0(x_0,x)$.
\end{proof}

\subsection{Horizontal and vertical index formulas}

There is a  first natural connection on $\M$ that respects the foliation structure, the Bott connection, which is given as follows:
\[
\nabla_X Y =
\begin{cases}
\pi_{\mathcal{H}} ( \nabla_X^g Y) , &X,Y \in \Gamma^\infty(\mathcal{H}), \\
\pi_{\mathcal{H}} ( [X,Y]),  &X \in \Gamma^\infty(\mathcal{V}),\ Y \in \Gamma^\infty(\mathcal{H}), \\
\pi_{\mathcal{V}} ( [X,Y]),  &X \in \Gamma^\infty(\mathcal{H}),\ Y \in \Gamma^\infty(\mathcal{V}), \\
\pi_{\mathcal{V}} ( \nabla_X^g Y) , &X,Y \in \Gamma^\infty(\mathcal{V}),
\end{cases}
\]
where $\nabla^g$ is the Levi-Civita connection for $g$ and $\pi_\mathcal{H}$ (resp.~$\pi_\mathcal{V}$) the projection on $\mathcal{H}$ (resp.~$\mathcal{V}$). It is easy to check that for every $\varepsilon >0$, this connection satisfies $\nabla g_\varepsilon=0$. A fundamental property of $\nabla$ is that $\Ho$ and $\V$ are parallel.

\

The torsion $T$ of $\nabla$ is given as
$$T(X,Y) = -\pi_{\V} [\pi_\Ho X, \pi_\Ho Y ].$$

For $Z \in \Gamma^\infty(\V)$, there is a  unique skew-symmetric endomorphism  $J_Z:\mathcal{H}_x \to \mathcal{H}_x$ such that for all horizontal vector fields $X$ and $Y$,
\begin{align}\label{Jmap}
g_\mathcal{H} (J_Z (X),Y)= g_\mathcal{V} (Z,T(X,Y)),
\end{align}
where $T$ is the torsion tensor of $\nabla$. We extend $J_{Z}$ to be 0 on  $\mathcal{V}_x$. Also, if $Z\in \Gamma^\infty (\Ho)$, from \eqref{Jmap} we set $J_Z=0$.

\

In the sequel, we shall make extensive use of the notion of adjoint connection. Adjoint connections naturally appear in the study of Weitzenb\"ock type identities (see \cite{Elworthy,GT3}). If~$D$ is a connection on $\M$, the adjoint connection of $D$ will be denoted $\hat{D}$ and is defined by 
\[
\hat{D}_X Y=D_X Y -T^D (X,Y)
\]
where $T^D$ is the torsion tensor of $D$. Metric connections whose adjoint connections are also metric are the natural generalizations of Levi-Civita connections (see \cite{GT3} and Appendix~1).

The adjoint connection of the Bott connection is not metric. For this reason, for computations, we shall rather make use of the following family of connections first introduced in~\cite{Bau}:
\[
\nabla^\varepsilon_X Y= \nabla_X Y -T(X,Y) +\frac{1}{\varepsilon} J_Y X,
\] 
and we shall only keep the Bott connection as a reference connection.
It is readily checked that $\nabla^\varepsilon g_\varepsilon =0$. The adjoint connection of $\nabla^\varepsilon$ is then given by
\[
\hat{\nabla}^\varepsilon_X Y=\nabla_X Y +\frac{1}{\varepsilon} J_X Y,
\]
thus $\hat{\nabla}^\varepsilon$ is also a metric connection. It moreover preserves the horizontal and vertical bundle, in contrast to the connection $\nabla^\ve$ which does not have this property.

\

For later use, we record that the torsion of $\hat{\nabla}^\varepsilon$ is 
\[
\hat{T}^\varepsilon (X,Y)=T(X,Y)-\frac{1}{\varepsilon} J_Y X+\frac{1}{\varepsilon} J_X Y.
\]
The Riemannian curvature tensor of $\hat{\nabla}^\varepsilon$ is easily computed as
\begin{align}\notag
\hat{R}^\varepsilon (X,Y)Z&=  R(X,Y)Z +\frac{1}{\varepsilon} J_{T(X,Y)} Z +\frac{1}{\varepsilon^2} (J_X J_Y -J_Y J_X)Z\\ &\quad+  \frac{1}{\varepsilon} (\nabla_X J)_Y  Z -  \frac{1}{\varepsilon} (\nabla_Y J)_X  Z
\label{curvature adjoint}\end{align}
where $R$ is the curvature tensor of the Bott connection.  

\

 Since $\nabla^\varepsilon$ and $\hat{\nabla}^\varepsilon$ are both metric,  observe that the Levi-Civita connection $ \nabla^{g_\ve}$ for the metric $g_\varepsilon$ is given by $\frac{1}{2} ( \nabla^\varepsilon +\hat{\nabla}^\varepsilon)$. In particular, one has:

\begin{equation} \label{Bott} \nabla_X Y = \nabla^{g_\ve}_X Y + \frac{1}{2} T(X,Y) - \frac{1}{2\ve} J_X Y - \frac{1}{2\ve} J_Y X.\end{equation}

We point out that working with  $\nabla^\varepsilon $ and  $\hat{\nabla}^\varepsilon $ instead of the Levi-Civita connection  $\nabla^{g_\ve}$ greatly simplifies some computations (see Remark \ref{Levi remark} and Section~\ref{sec:HTypeComp}), whereas we can still freely use simple  second variation formulas (see Appendix~1).

\

The following lemma is obvious.

\begin{lemma}[Geodesic equation]
The equation for $g_\varepsilon$-geodesics is
\[
\nabla_{\gamma'} \gamma'+\frac{1}{\varepsilon} J_{\gamma'} \gamma' =0.
\]
\end{lemma}

\begin{proof}
The equation for $g_\varepsilon$ geodesics is $\nabla^{g_\ve}_{\gamma'} \gamma' =0$, and one concludes with \eqref{Bott}.
\end{proof}

After these preliminaries, we are now ready to prove one of the main results of the section.  As before, let $d_\ve$ be the distance of $g_\ve$. Let $x_0 \in \M$ be any point and define $r_\ve = r_{\ve,x_0}$ by $r_\ve(x) = d_\ve(x_0, x)$.

\begin{proposition}[Horizontal and vertical index formulas]\label{index formulas}
Let $\nabla^2$ denote the Hessian of the Bott connection $\nabla$. If $x$ is not in the cut-locus of $x_0$ with respect to $g_\ve$, and if $\gamma$ is the unique $g_\varepsilon$ geodesic from $x_0$ to $x$ parametrized by arc length, then:
\begin{enumerate}[\rm (a)]
\item For every $v \in \Ho_x$ and vector field $Y$ along $\gamma$, taking values in $\Ho$ and satisfying $Y(0) = 0$ and $Y(r_\ve(x)) = v$, we have
$$\nabla^2 r_\ve(v,v) \leq I_{\Ho,\ve}(Y, Y)$$
where
\begin{align*} I_{\Ho,\ve}(Y, Y) & = \int_0^{r_\ve(x)} \left( \|\hat{\nabla}^{2\ve}_{\gamma'} Y \|_g^2(t) + \langle R( \gamma', Y )  \gamma', Y \rangle_g(t)  \right) dt \\
& \quad + \frac{1}{\ve} \int_0^{r_\ve(x)} \left( \langle (\nabla_{Y} T)(Y,  \gamma') ,  \gamma' \rangle_g(t) + \| T( \gamma', Y) \|^2_g(t) - \frac{1}{4\ve} \| J_{ \gamma'} Y\|^2_g(t)\right) dt. 
\end{align*}
\item For every $w \in \V_x$ and vector field $Z$ along $\gamma$, taking values in $\V$ and satisfying $Z(0) = 0$ and $Z(r_\ve(x)) = w$, we have
$$\nabla^2 r_\ve (w,w) \leq I_{\V,\ve}(Z, Z)$$
where
$$I_{\V,\ve}(Z, Z) = \frac{1}{\ve} \int_0^{r_\ve(x)} \left( \|\nabla_{ \gamma'} Z \|_g^2(t) + \langle R( \gamma', Z )  \gamma', Z \rangle_g (t)   \right) dt.$$
\end{enumerate}

\end{proposition}

\begin{proof}
We prove a). The proof of b) follows by a similar and even simpler computation. From the classical theory (see Lemma \ref{index lemma} in the Appendix) one has:
$$\nabla^2 r_\ve(v,v) \leq I_{\Ho,\ve}(Y, Y),$$
where
\[
 I_{\Ho,\ve}(Y, Y)=\int_0^{r_\varepsilon} \left( \| \hat{\nabla}^\varepsilon_{\gamma'} Y  \|_\varepsilon^2-\langle \hat{R}^\varepsilon(\gamma',Y)Y,\gamma'\rangle_\varepsilon -\langle \hat{T}^\varepsilon(Y,\hat{\nabla}^\varepsilon_{\gamma'} Y),\gamma' \rangle_\varepsilon \right) dt.
\]

Since $Y$ is horizontal, one has
\[
\hat{R}^\varepsilon(\gamma',Y)Y=R(\gamma',Y)Y+\frac{1}{\varepsilon} J_{T(\gamma',Y)} Y-  \frac{1}{\varepsilon} (\nabla_{Y} J)_{\gamma'}  Y
\]
and
\[
\hat{T}^\varepsilon(Y,\hat{\nabla}^\varepsilon_{\gamma'} Y)=T(Y,\hat{\nabla}^\varepsilon_{\gamma'} Y).
\]
In particular, we deduce that
\[
\langle \hat{T}^\varepsilon(Y,\hat{\nabla}^\varepsilon_{\gamma'} Y),\gamma' \rangle_\varepsilon=\frac{1}{\varepsilon} \langle J_{\gamma'} Y , \nabla_{\gamma'} Y \rangle_{\Ho} +\frac{1}{\varepsilon^2} \| J_{\gamma'} Y\|^2_{\Ho}.
\]
To conclude the proof, we observe that
\begin{align*}
\| \hat{\nabla}^\varepsilon_{\gamma'} Y  \|_\varepsilon^2-\frac{1}{\varepsilon} \langle J_{\gamma'} Y , \nabla_{\gamma'} Y \rangle_{\Ho}&=\| \nabla_{\gamma'} Y  +\frac{1}{\varepsilon}  J_{\gamma'} Y \|_{\Ho}^2-\frac{1}{\varepsilon} \langle J_{\gamma'} Y , \nabla_{\gamma'} Y \rangle_{\Ho} \\
 &=\| \nabla_{\gamma'} Y  +\frac{1}{2\varepsilon}  J_{\gamma'} Y \|_{\Ho}^2+\frac{3}{4\varepsilon^2} \|J_{\gamma'} Y \|_\varepsilon^2.\qedhere
\end{align*}
\end{proof}

\begin{remark}\label{Levi remark}
By using \eqref{Bott}, a lengthy but routine computation shows  that the Riemannian curvature tensor of the Levi-Civita connection is given by
\begin{align} \label{LCBottCurv} & R^{g_\ve}(X,Y) Z \\ \nonumber
& = R(X,Y) Z - \frac{1}{2} (\nabla_X T)(Y,Z) + \frac{1}{2} (\nabla_Y T)(X,Z)  + \frac{1}{2\ve} (\nabla_X J)_Y Z - \frac{1}{2\ve} (\nabla_Y J)_X Z\\ \nonumber
& \quad    + \frac{1}{2\ve} (\nabla_X J)_Z Y - \frac{1}{2\ve} (\nabla_Y J)_Z X + \frac{1}{2\ve} J_{T(X,Y)} Z, \\ \nonumber
& \quad - \frac{1}{4\ve} T(X, J_Y Z + J_Z Y)   + \frac{1}{4\ve^2} J_X \left(J_Y Z + J_Z Y\right) - \frac{1}{4\ve} J_{T(Y,Z)} X \\ \nonumber
& \quad + \frac{1}{4\ve} T(Y, J_X Z + J_Z X)   - \frac{1}{4\ve^2} J_Y \left(J_X Z + J_Z X\right) + \frac{1}{4\ve} J_{T(X,Z)} Y.
\end{align}
Using this formula and the usual index formulas for the Levi-Civita connection  yield the same horizontal and vertical index formulas, however using the adjoint connection $\hat{\nabla}^\varepsilon$ greatly simplifies computations.
\end{remark}

\subsection{Horizontal Laplacian comparison theorem}
\label{sec:H-Lc}

We now introduce the relevant tensors which will be used to control the index forms.  The horizontal divergence of the torsion $T$ is the $(1,1)$ tensor  which in a local horizontal frame $X_1, \dots, X_n$  is defined by
\[
\delta_\mathcal{H} T (X):= -\sum_{j=1}^n(\nabla_{X_j} T) (X_j,X).
\]
Going forward, we will always assume in the sequel of the paper that the horizontal distribution $\Ho$ satisfies \emph{the Yang-Mills condition}, meaning that $\delta_\Ho T = 0$ (see \cite{B2,GT1,GT2} for the geometric significance of this condition).

\

We will denote by $\mathbf{Ric}_\mathcal{H}$ the horizontal Ricci curvature of the Bott connection, that is to say the horizontal trace of the full curvature tensor $R$  of the Bott connection. Using the observation that $\nabla$ preserves the splitting $\Ho \oplus \V$ and from the first Bianchi identity, it follows that $\mathbf{Ric}_\Ho( X, Y) = \mathbf{Ric}_\Ho(\pi_\Ho X,\pi_\Ho Y)$ (see the computation in the proof of Lemma \ref{positive vertical} for details).

\

If $Z_1,\dots,Z_m$ is a local vertical frame, the $(1,1)$ tensor
\[
\mathbf{J}^2 := \sum_{\ee=1}^m J_{Z_\ee}J_{Z_\ee}
\]
does not depend on the choice of the frame and may globally be defined.

\begin{remark}\label{ricci}
A simple computation (see for instance Theorem 9.70, Chapter 9 in \textup{\cite{Besse}}) gives the following result for the Riemannian Ricci curvature of the metric $g_\varepsilon$. For every $X \in \Gamma^\infty(\mathcal{H})$ and $Z \in \Gamma^\infty(\mathcal{V})$,
\begin{align*}
&\mathbf{Ric}^{g_\varepsilon} (Z,Z)=\mathbf{Ric}_\mathcal{V} (Z,Z)-\frac{1}{4\varepsilon^2} \mathbf{Tr} (J^2_{Z})\\
&\mathbf{Ric}^{g_\varepsilon} (X,Z)=0\\
&\mathbf{Ric}^{g_\varepsilon} (X,X)=\mathbf{Ric}_\mathcal{H} (X,X)+\frac{1}{2\varepsilon} \langle \mathbf{J}^2 X, X  \rangle_\mathcal{H} ,
\end{align*}
where $\mathbf{Ric}_\V $ is   the Ricci curvature of the leaves as sub-manifolds of $(\M,g)$.
\end{remark}

Let $x_0\in \M$ be fixed and for $\varepsilon >0$ let
\[
r_\varepsilon (x) =d_\varepsilon (x_0,x).
\]
We assume  that globally on $\M$, for every $X \in \Gamma^\infty(\mathcal{H})$ and $Z \in \Gamma^\infty(\mathcal{V})$,
\[
\mathbf{Ric}_\Ho (X,X) \ge  \rho_1 (r_\varepsilon )  \| X \|^2_{\mathcal{H}}, \quad -\langle \mathbf{J}^2 X, X  \rangle_\mathcal{H} \le \kappa (r_\varepsilon) \| X \|^2_\mathcal{H},\quad -\frac{1}{4} \mathbf{Tr} (J^2_{Z})\ge \rho_2(r_\varepsilon) \| Z \|^2_\mathcal{V},
\]
for some continuous functions $\rho_1,\rho_2,\kappa$.

\begin{theorem}\label{comparison 1}
Consider the operator $\Delta_\ch = \mathbf{Tr}_\ch \nabla^2 = \mathrm{div} \, \nabla_\ch$.
Let $x \in \M$, $x \neq x_0$ and $x$ not in the $d_\varepsilon$ cut-locus of $x_0$. Let $G:[0,r_\varepsilon(x)]\to \mathbb{R}_{\ge 0}$ be a differentiable function which is positive on $(0,r_\varepsilon(x)]$ and such that $ G(0)=0$. We have
\[
\Delta_\Ho r_\varepsilon(x) \le \frac{1}{G(r_\varepsilon(x))^2} \int_0^{r_\varepsilon(x)} \left(nG'(s)^2-\left[ \left(\rho_1 (s )-\frac{1}{\varepsilon} \kappa(s)  \right)\Gamma(r_\varepsilon)(x)+\rho_2(s)  \Gamma^\V(r_\varepsilon)(x)\right]G(s)^2 \right)ds.
\]
\end{theorem}

\begin{proof}
Let $\gamma$ be the unique length parametrized $g_\varepsilon$-geodesic between $x_0$ and $x$. Let $X_1,\ldots,X_n$ be a  horizontal orthonormal frame along $\gamma$ such that
\[
\nabla_{\gamma'}X_i+\frac{1}{2\varepsilon} J_{\gamma'} X_i = \hat \nabla^{2\ve}_{\gamma'} X_i =0.
\]
We have
\[
\Delta_\Ho r_\varepsilon(x)=\sum_{i=1}^n \nabla^2  r_\varepsilon (X_i,X_i).
\]
Consider now the vector fields along $\gamma$ defined by
\[
Y_i=\frac{G( s)}{G(r_\varepsilon (x))} X_i, \quad 0\le s \le r_\varepsilon (x).
\]
Using Proposition \ref{index formulas} and the Yang-Mills condition  one finds
\begin{align*}
 & \sum_{i=1}^n I_{\mathcal{H}, \varepsilon} (\gamma,Y_i,Y_i)\\
 &=  \sum_{i=1}^n \int_0^{r_\varepsilon(x)} \left(\| \nabla_{\gamma'}Y_i+\frac{1}{2\varepsilon} J_{\gamma'} Y_i \|_\varepsilon^2 -\langle R(\gamma_\Ho',Y_i)Y_i,\gamma_\Ho'\rangle_\varepsilon+\|T(Y_i,\gamma_\Ho') \|^2_\varepsilon-\frac{1}{4\varepsilon^2} \|J_{\gamma_\V'} Y_i \|_\varepsilon^2 \right) dt \\
 &= \frac{1}{G(r_\varepsilon(x))^2} \int_0^{r_\varepsilon(x)} nG'(s)^2 +\sum_{i=1}^n  G(s)^2 \left( -\langle R(\gamma_\Ho',X_i)X_i,\gamma_\Ho'\rangle_\varepsilon+\|T(X_i,\gamma_\Ho') \|^2_\varepsilon-\frac{1}{4\varepsilon^2} \|J_{\gamma_\V'} X_i \|_\varepsilon^2\right) ds \\
 &= \frac{1}{G(r_\varepsilon(x))^2} \int_0^{r_\varepsilon(x)} nG'(s)^2 +  G(s)^2 \left(  -\mathbf{Ric}_\Ho (\gamma'_\Ho,\gamma'_\Ho) -\frac{1}{\varepsilon} \langle \mathbf{J}^2 \gamma'_\Ho , \gamma'_\Ho \rangle_\Ho+\frac{1}{4\varepsilon^2}\mathbf{Tr}_\mathcal{H} (J^2_{\gamma'_\V})\right) ds \\
 &\le  \frac{1}{G(r_\varepsilon(x))^2} \int_0^{r_\varepsilon(x)} \left(nG'(s)^2-\left[ \left(\rho_1 (s )-\frac{1}{\varepsilon} \kappa(s)  \right) \| \gamma'(s) \|^2_\mathcal{H} +\frac{\rho_2(s)}{\varepsilon^2}   \| \gamma'(s) \|^2_\mathcal{V} \right]G(s)^2 \right)ds.
\end{align*}
From \eqref{observation}, one has
\[
\| \gamma'(s) \|^2_\mathcal{H} =\Gamma(r_\varepsilon)(x), \quad \| \gamma'(s) \|^2_\mathcal{V} =\varepsilon^2 \Gamma^\mathcal{V}(r_\varepsilon)(x),
\]
which completes the proof.
\end{proof}

\begin{remark} \label{rem:sLc}
Since $\Gamma(r_\varepsilon)+\varepsilon \Gamma^\V(r_\varepsilon) =1$, one can rewrite the previous inequality as
\[
\Delta_\Ho r_\varepsilon(x) \le \frac{1}{G(r_\varepsilon(x))^2} \int_0^{r_\varepsilon(x)} \left(nG'(s)^2-\left[ \left(\rho_1 (s )-\frac{1}{\varepsilon}( \kappa(s)+\rho_2(s))  \right)\Gamma(r_\varepsilon)(x)+\frac{\rho_2(s)}{\varepsilon} \right]G(s)^2 \right)ds.
\]
\end{remark}

Optimizing the function $G$ in the previous theorem when $\rho_1,\rho_2, \kappa$ are constants yields:

\begin{corollary}\label{comparison constant}
Assume that the functions $\rho_1,\kappa,\rho_2$ are constant. Then, for $x\neq x_0$ not in the $d_\varepsilon$ cut-locus of $x_0$,
\[
\Delta_\Ho r_\varepsilon(x) \le F( r_\varepsilon(x), \Gamma(r_\varepsilon)(x)),
\]
where
\begin{align*}
F(r,\gamma)=
\begin{cases}
 \sqrt { n \kappa_\varepsilon (\gamma)} \cot(\sqrt { \frac{\kappa_\varepsilon (\gamma)}{n}} r ), &\text{if}\ \kappa_{\varepsilon} (\gamma)>0,
\\
\displaystyle\frac{n}{r}, &\text{if}\ \kappa_{\varepsilon}(\gamma) = 0,
\\
 \sqrt{n |\kappa_\varepsilon (\gamma)|} \coth(\sqrt{\frac{|\kappa_\varepsilon (\gamma)|}{n}} r), &\text{if}\ \kappa_{\varepsilon}(\gamma) <0,
\end{cases}
\end{align*}
and 
\[
\kappa_\varepsilon (\gamma) =\left(\rho_1 -\frac{1}{\varepsilon}( \kappa+\rho_2)  \right)\gamma+\frac{\rho_2}{\varepsilon}, \quad \gamma \in [0,1].
\]
\end{corollary}

In particular, since $\Gamma(r_\varepsilon)$ is always between 0 and 1, we get:

\begin{corollary}\label{comparison constant 2}
Assume that the functions $\rho_1,\kappa,\rho_2$ are constant.  Denote
\[
\kappa_\varepsilon= \min \left(\rho_1-\frac{\kappa}{\varepsilon}, \frac{\rho_2} {\varepsilon} \right).
\]
For $x\neq x_0 \in \M $, not in the $d_\varepsilon$ cut-locus of $x_0$
\begin{equation}\label{lc2}
\Delta_{\mathcal{H}} r_{\varepsilon}(x) \le \begin{cases}
 \sqrt {n\kappa_\varepsilon} \cot(\sqrt {\frac{\kappa_\varepsilon}{n}} r_{\varepsilon}(x)),&\text{if}\ \kappa_{\varepsilon} >0,
\\
\displaystyle\frac{n}{r_{\varepsilon}(x)},&\text{if}\ \kappa_{\varepsilon} = 0,
\\
 \sqrt{n|\kappa_\varepsilon|} \coth(\sqrt{\frac{|\kappa_\varepsilon|}{n}} r_{\varepsilon}(x)),&\text{if}\ \kappa_{\varepsilon} <0.
\end{cases}
\end{equation}
\end{corollary}

We conclude this section with two easy corollaries from our horizontal Laplacian comparison theorem.

\begin{corollary}[Injectivity radius estimate]\label{injectivity}
Assume that the functions $\rho_1,\kappa,\rho_2$ are constant with $\rho_2 >0$. Then, for $x_0 \in \M$ the $d_\varepsilon$ distance of $x_0$ to its cut-locus is less than $\pi \sqrt{\frac{n\varepsilon}{\rho_2} }$.
\end{corollary}

\begin{proof}
Let $x_0 \in \M$. Let us denote by $\mathcal{L}_{x_0}$ the leaf going through $x_0$ and consider a $g_\varepsilon$ length parametrized geodesic $\gamma$ in $\mathcal{L}_{x_0}$ such that $\gamma(0)=x_0$. From Corollary \ref{comparison constant} one has
\[
\Delta_\Ho r_\varepsilon( \gamma(s)) \le  \sqrt { \frac{n \rho_2}{\varepsilon} } \cot \left(\sqrt { \frac{\rho_2}{n\varepsilon}} s \right).
\]
One deduces that
\[
\lim_{s \to \pi \sqrt{\frac{n\varepsilon}{\rho_2} }} \Delta_\Ho r_\varepsilon( \gamma(s))=-\infty.
\]
Therefore, $ r_\varepsilon (\gamma(s))$ can not be differentiable at  $s = \pi \sqrt{\frac{n\varepsilon}{\rho_2} }$. We deduce that  the $d_\varepsilon$ distance of $x_0$ to its cut-locus is less than $\pi \sqrt{\frac{n\varepsilon}{\rho_2} }$. 
\end{proof}

\begin{remark}
A first version of this theorem is proved in \cite{BD} in the case where the foliation is the Reeb foliation of a Sasakian manifold (in that case $\rho_2=\frac{n}{4}$).
\end{remark}

\begin{corollary}[Bonnet-Myers type theorem]\label{BMyers}
Assume that the functions $\rho_1,\kappa,\rho_2$ are constant with $\rho_1, \rho_2 >0$, then $\M$ is compact. Moreover, for $\varepsilon > \frac{\kappa}{\rho_1}$,
\[
\mathbf{diam} ( \M , d_\varepsilon) \le \pi \sqrt{ \frac{n}{\kappa_\varepsilon}},
\]
where
\[
\kappa_\varepsilon= \min \left(\rho_1-\frac{\kappa}{\varepsilon}, \frac{\rho_2} {\varepsilon} \right).
\]
\end{corollary}

\begin{proof}
Let $x_0 \in \M$ and $\varepsilon > \frac{\kappa}{\rho_1}$. From Corollary \ref{comparison constant 2}, one has for $x\neq x_0  $, not in the $d_\varepsilon$ cut-locus of $x_0$
\begin{equation*}
\Delta_{\mathcal{H}} r_{\varepsilon}(x) \le \sqrt {n\kappa_\varepsilon} \cot \left(\sqrt {\frac{\kappa_\varepsilon}{n}} r_{\varepsilon}(x)\right).
\end{equation*}
We deduce from Calabi's lemma that any point $x$ such that $d_\varepsilon (x_0,x) \ge \pi \sqrt{ \frac{n}{\kappa_\varepsilon}}$ has to be in the cut-locus of $x_0$.
Let now $x \in \M$ arbitrary. If $x$ is not in the cut-locus of $x_0$, then $d(x_0,x) < \pi \sqrt{ \frac{n}{\kappa_\varepsilon}}$. If $x$ is in the cut-locus of $x_0$ then for every $\eta >0$ there is at least one point $y$ in the open ball with center $x$ and radius $\eta$ such that $y$ is not in the cut-locus of $x_0$. Thus $d(x_0 , x)\le \pi \sqrt{ \frac{n}{\kappa_\varepsilon}} +\eta$.
\end{proof}

\begin{remark}
The first Bonnet-Myers theorem in that situation was proved in \textup{\cite{BG}} by using heat equation methods.  It was proved that the sub-Riemannian diameter of $\M$ satisfies the bound
\begin{align*} 
\mathbf{diam} ( \M , d_0) \le 2\sqrt{3} \pi \sqrt{
\frac{\kappa+\rho_2}{\rho_1\rho_2} \left(
1+\frac{3\kappa}{2\rho_2}\right)n }.
\end{align*}

\end{remark}
\subsection{Vertical Laplacian comparison theorem}

Let $x_0\in \M$ be fixed and, as before, for $\varepsilon >0$ denote
\[
r_\varepsilon (x) =d_\varepsilon (x_0,x).
\]
We assume  that globally on $\M$, for every $Z \in \Gamma^\infty(\mathcal{V})$,
\[
\mathbf{Ric}_\V (Z,Z) \ge  \rho_3 (r_\varepsilon )  \| Z \|^2_{\mathcal{V}},
\]
where $\mathbf{Ric}_\V $ is the vertical Ricci curvature of the Bott connection (this is also the Ricci curvature of the leaves of the foliation as sub-manifolds of $(\M,g)$), and where $\rho_3$ is some continuous function.

\begin{theorem}\label{vertical comparison}
Let $x \in \M$, $x \neq x_0$, not in the $d_\varepsilon$ cut-locus of $x_0$. Let $G:[0,r_\varepsilon(x)]\to \mathbb{R}_{\ge 0}$ be a differentiable function which is positive on $(0,r_\varepsilon(x)]$ and such that $ G(0)=0$. We have
\[
\Delta_\V r_\varepsilon(x) \le \frac{1}{G(r_\varepsilon(x))^2} \int_0^{r_\varepsilon(x)} \left(\frac{m}{\varepsilon}G'(s)^2-\rho_3(s) \varepsilon \Gamma^\mathcal{V}(r_\varepsilon)(x)G(s)^2 \right)ds.
\]
\end{theorem}

The proof is similar to that of Theorem~\ref{comparison 1}. As an immediate corollary we deduce:

\begin{corollary}\label{vertical comparison 2}
Assume that the function $\rho_3$ is constant. Then, for $x \neq x_0$ not in the cut-locus  of $x_0$,
\[
\Delta_\V r_\varepsilon(x) \le  F(r_\varepsilon (x),\Gamma^\mathcal{V}(r_\varepsilon)(x))
\]
where
\begin{align*}
F(r_\varepsilon,\Gamma^\mathcal{V}(r_\varepsilon))=
\begin{cases}
 \sqrt {m \rho_3  \Gamma^\mathcal{V}(r_\varepsilon)} \cot\left(\sqrt { \frac{\rho_3\varepsilon^2 \Gamma^\mathcal{V}(r_\varepsilon)}{m}} r_\varepsilon \right),&\text{if} \quad  \rho_3 >0,
\\
\displaystyle\frac{m}{\varepsilon r_\varepsilon},&\text{if} \quad \rho_3 = 0,
\\
 \sqrt {- m\rho_3  \Gamma^\mathcal{V}(r_\varepsilon)} \coth\left(\sqrt {- \frac{\rho_3\varepsilon^2 \Gamma^\mathcal{V}(r_\varepsilon)}{m}} r_\varepsilon \right),&\text{if} \quad \rho_3 <0.
\end{cases}
\end{align*}
\end{corollary}

\subsection{Horizontal and vertical Bochner formulas and Laplacian comparison theorems}
\label{ssec:HorVerComp}

It is well-known that on Riemannian manifolds the Laplacian comparison theorem may also be obtained as a consequence of the Bochner formula. In this section, we show that Theorems \ref{comparison 1} and \ref{vertical comparison} may also be obtained as a consequence of Bochner type identities. The methods developed in the previous sections are more powerful to understand second derivatives of the distance functions (see Section~\ref{sec:HTypeComp}), but using Bochner type identities and the resulting curvature dimension estimates has the advantage to be applicable in more general situations (see \cite{BG,GT1,GT2} for the general framework on curvature dimension inequalities).

\

We first recall the horizontal and vertical Bochner identities that were respectively proved in \cite{BKW} and \cite{BB} (see also \cite{GT1,GT2} for generalizations going beyond the foliation case).

\begin{theorem}[Horizontal and vertical Bochner identities]\label{Bochner2}
For $f \in C^\infty (\M)$, one has
\[
\frac{1}{2} \Dh \| df \|_{\varepsilon}^2 -\langle d \Dh f , df \rangle_{\varepsilon} =  \| \nabla^\varepsilon_{\mathcal{H}} df  \|_{\varepsilon}^2 + \left\langle \mathbf{Ric}_{\mathcal{H}} (df), df \right\rangle_\mathcal{H} +\frac{1}{\varepsilon} \langle \mathbf{J}^2 (df) , df \rangle_\mathcal{H},
\]
and
\[
\frac{1}{2} \Dv \| df \|_{\varepsilon}^2 -\langle d \Dv f , df \rangle_{\varepsilon} =  \| \nabla_{\Ho,\V}^2 f \|^2+\varepsilon  \| \nabla_\V^2 f \|^2 +\varepsilon  \left\langle \mathbf{Ric}_{\mathcal{V}} (df), df \right\rangle_\mathcal{V}.
\]

\end{theorem}

\begin{proof}
The first identity is Theorem 3.1 in \cite{BKW}. The second identity may derived from Proposition 2.2 in \cite{BB}.
\end{proof}

Those two Bochner formulas may be used to prove general curvature dimension estimates respectively for the horizontal and vertical Laplacian.

We introduce the following operators defined for $f,g \in C^\infty(\M)$,
\begin{align*}
&\Gamma(f,g)=\frac{1}{2} ( \Dh (fg) -g\Dh f-f\Dh g)=\langle \nabla_\mathcal{H} f , \nabla_\mathcal{H} g\rangle_\mathcal{H},\\
&\Gamma^\mathcal{V} (f,g)=\langle \nabla_\mathcal{V} f , \nabla_\mathcal{V} g\rangle_\mathcal{V},
\end{align*}
and their iterations which are defined by
\begin{align*}
\Gamma^\mathcal{H}_2(f,g)&=\frac{1}{2} ( \Dh(\Gamma(f,g)) -\Gamma(g,\Dh f)-\Gamma(f,\Dh g))\\
\Gamma^{\Ho,\mathcal{V}}_2(f,g)&=\frac{1}{2} ( \Dh (\Gamma^\mathcal{V}(f,g)) -\Gamma^\mathcal{V}(g,\Dh f)-\Gamma^\mathcal{V}(f,\Dh g))\\
\Gamma^{\mathcal{V},\Ho}_2(f,g)&=\frac{1}{2} ( \Dv (\Gamma(f,g)) -\Gamma(g,\Dv f)-\Gamma(f,\Dv g))
\end{align*}
and
\[
\Gamma^\mathcal{V}_2(f,g)=\frac{1}{2} ( \Dv(\Gamma^\V(f,g)) -\Gamma^\V(g,\Dv f)-\Gamma^\V(f,\Dv g))
\]
As a straightforward consequence of Theorem \ref{Bochner2}, we obtain the following generalized curvature dimension inequalities for the horizontal and vertical Laplacians.

\begin{theorem}\label{CD}\ 

\begin{enumerate}[\rm 1.]
\item Assume that globally on $\M$, for every $X \in \Gamma^\infty(\mathcal{H})$ and $Z \in \Gamma^\infty(\mathcal{V})$,
\[
\mathbf{Ric}_\Ho (X,X) \ge  \rho_1 (r_\varepsilon )  \| X \|^2_{\mathcal{H}}, \quad -\langle \mathbf{J}^2 X, X  \rangle_\mathcal{H} \le \kappa (r_\varepsilon) \| X \|^2_\mathcal{H},\quad -\frac{1}{4} \mathbf{Tr}_\mathcal{H} (J^2_{Z})\ge \rho_2(r_\varepsilon) \| Z \|^2_\mathcal{V},
\]
for some continuous functions $\rho_1,\rho_2,\kappa$. For every $f \in C^\infty(\M)$, one has
\[
\Gamma^\Ho_2(f,f)+\varepsilon \Gamma^{\Ho,\mathcal{V}}_2(f,f)\ge \frac{1}{n} (\Dh f)^2 +\left( \rho_1 (r_\varepsilon) -\frac{\kappa(r_\varepsilon)}{\varepsilon}\right) \Gamma(f,f)+\rho_2 (r_\varepsilon) \Gamma^\mathcal{V} (f,f).
\]
\item Assume  that globally on $\M$, for every $Z \in \Gamma^\infty(\mathcal{V})$,
\[
\mathbf{Ric}_\V (Z,Z) \ge  \rho_3 (r_\varepsilon )  \| Z \|^2_{\mathcal{V}}, 
\]
for some continuous functions $\rho_3$. For every $f \in C^\infty(\M)$ one has
\[
 \Gamma^{\mathcal{V},\Ho}_2(f,f)+\varepsilon \Gamma^\V_2(f,f)\ge \frac{\varepsilon}{m} (\Dv f)^2 +\varepsilon \rho_3 (r_\varepsilon) \Gamma^\mathcal{V} (f,f).
\]
\end{enumerate}
\end{theorem}

\begin{proof}
The proof of 1. follows from 
\begin{align*}
  \| \nabla^\varepsilon_{\mathcal{H}} df  \|_{\varepsilon}^2
   &  \ge \| \nabla^2_\Ho f  \|^2
  -\frac{1}{4}\mathbf{Tr}_\mathcal{H} (J^2_{df}) \\
   & \ge \frac{1}{n} (\Dh f)^2+\rho_2 (r_\varepsilon) \Gamma^\mathcal{V} (f,f),
  \end{align*}
  where we refer to the proof of Theorem 3.1 in \cite{BKW} for the details. The proof of 2.~is immediate.
\end{proof}

For an alternative proof of Theorem  \ref{comparison 1} and \ref{vertical comparison}, we shall need the easily proved  following  lemma.

\begin{lemma}\label{zero}
We have
\[
\lim_{x \to x_0} r_\varepsilon (x)^2\Dh r_\varepsilon (x)=\lim_{x \to x_0} r_\varepsilon (x)^2\Dv r_\varepsilon (x)=0.
\]
\end{lemma}

We are now in position to give a second proof of Theorem  \ref{comparison 1}.

\begin{proof}[Proof \rm\textbf{(Second proof of Theorem~\ref{comparison 1})}]
Let $\gamma(t)$, $0\le t \le r_\varepsilon(x)$, be the unique length parametrized $g_\varepsilon$-geodesic between $x_0$ and $x$. We denote
\[
\phi (t) =\Delta_\ch r_\varepsilon (\gamma(t)), \quad 0 < t \le r_\varepsilon(x).
\]
From Theorem \ref{CD}, we get the differential inequality
\begin{align}\label{diff in}
-\phi'(t) \ge \frac{1}{n} (\phi(t))^2 +\left( \rho_1 (t) -\frac{\kappa(t)}{\varepsilon}\right) \Gamma(r_\varepsilon)(x)+\rho_2 (t) \Gamma^\mathcal{V} (r_\varepsilon)(x),
\end{align}
because  $\Gamma(r_\varepsilon)$ and $\Gamma^\mathcal{V} (r_\varepsilon)$ are constants along $\gamma$. We now notice the lower bound
\[
\frac{1}{n} (\phi(t))^2 \ge 2 \frac{G'(t)}{G(t)} \phi(t) -  n \frac{G'(t)^2}{G(t)^2}.
\]
Using this lower bound in \eqref{diff in}, multiplying by $G(t)^2$, and integrating from $0$ to $r_\varepsilon (x)$ yields the expected result thanks to lemma \ref{zero}.
\end{proof}

The second proof of Theorem \ref{vertical comparison} is identical.

\section{Horizontal and vertical Hessian and Laplacian comparison theorems on Sasakian foliations} \label{sec:HTypeComp}

It is remarkable that Theorem \ref{comparison 1} does not require any assumption on the dimension or curvature of the vertical bundle. However, when $\varepsilon$ goes to 0 the upper bound for $\Delta_{\Ho} r_\varepsilon$ blows up to $\infty$, whereas it is known that in some situations one may expect a horizontal Laplacian comparison theorem for the sub-Riemannian distance $d_0$. Indeed, for instance in the 3-dimensional Heisenberg group it is known that in the distributional sense
\[
\Delta_\Ho r_0 \le \frac{4}{r_0}
\]
where $r_0$ is the distance to a fixed point, and the constant $4$ is sharp. This horizontal Laplacian comparison theorem for the sub-Riemannian distance has been first generalized in 3-dimensional Sasakian manifolds by Agrachev-Lee \cite{AL1}. See also a version proved in higher dimensional Sasakian manifolds by Lee-Li \cite{LL}, but note that this work contains some typos.

\

Inspired by some of the results in \cite{Lee} and \cite{Rifford}, we prove in this section that for Sasakian manifolds, a comparison theorem for the sub-Riemannian distance may be obtained as a limit when $\varepsilon \to 0$ of a comparison theorem for the distances $r_\varepsilon$.  With respect to \cite{AL1,Lee,LL}, we obtain an explicit and simple upper bound for $\Delta_\Ho r_\varepsilon $ which is sharp when $\varepsilon \to 0$, and in our opinion the method and computations are more straightforward and shorter. Our method has also the advantage to easily yield a Hessian comparison theorem  for the  distance $r_\varepsilon$, $\varepsilon > 0$ (such Hessian comparison theorem is not explicitly worked out in \cite{LL}) and a vertical Laplacian comparison theorem (see Theorem \ref{VL}).

\

We now describe the setting of Sasakian manifolds (see \cite{BoG} for further details about Sasakian geometry). Let $(\M, \theta,g)$ be a complete K-contact Riemannian manifold with Reeb vector field $S$. The Bott connection coincides  with the Tanno's connection that was introduced in \cite{Tanno} and which is the unique connection that satisfies:
\begin{enumerate}
\item $\nabla\theta=0$;
\item $\nabla S=0$;
\item $\nabla g=0$;
\item ${T}(X,Y)=d\theta(X,Y)S$ for any $X,Y\in \Gamma^\infty(\mathcal{H})$;
\item ${T}(S,X)=0$ for any vector field $X\in \Gamma^\infty(\mathcal{H})$.
\end{enumerate}
It is easy to see that the Reeb foliation is of Yang-Mills type if and only if $\delta_\Ho d \theta=0$. Equivalently, if we introduce an operator $\mathbf{J} : = J_S$, this condition writes $\delta_\Ho \mathbf J =0$. If $\M$ is a strongly pseudo convex CR manifold with pseudo-Hermitian form $\theta$, then the Tanno's connection is the Tanaka-Webster connection. In that case, we have then $\nabla \mathbf J=0$ (see \cite{Dragomir}) and thus  $\delta_\Ho \mathbf J =0$. CR manifold of K-contact type are called Sasakian manifolds (see \cite{Dragomir}). Thus, the Reeb foliation on any Sasakian manifold is of Yang-Mills type.
\

Throughout the section, we assume that the Riemannian foliation on $\M$ is the Reeb foliation of a Sasakian structure. The Reeb vector field on $\M$ will be denoted by $S$ and the complex structure by $\mathbf{J}$. The torsion of the Bott connection is then
\[
T(X,Y)=\langle \mathbf J X, Y \rangle_\Ho S.
\]
Therefore with the previous notations, one has 
\[
J_Z X =\langle Z,S \rangle \mathbf J X.
\]
In this setting, the formula \eqref{curvature adjoint} for the curvature of the adjoint connection greatly simplifies:
\begin{align*}
\hat{R}^\varepsilon (X,Y)Z & =  R(X,Y)Z +\frac{1}{\varepsilon} J_{T(X,Y)} Z \\
 &=R(X,Y)Z +\frac{1}{\varepsilon} \langle \mathbf J X, Y \rangle_\Ho \mathbf{J} Z.
\end{align*}

In a Sasakian space, for every non-vanishing horizontal vector field $X$,  $T\M$ is always generated by $[X,\Ho]$ and $\Ho$. Therefore the sub-Riemannian structure on a Sasakian foliation is fat (see~\textup{\cite{RS}} for a detailed discussion of such structures). In particular all sub-Riemannian geodesics are normal and from Corollary 6.1 in \textup{\cite{RT}}, for every $x_0 \in \M$, the distance function $x \to r_0 (x)$ is locally semi-concave in $\M\setminus \{ x_0 \}$. In particular, it is twice differentiable almost everywhere. Also, from Corollary 32 in \textup{\cite{BR2}}, $x \neq x_0$ is in $\mathbf{Cut}_0 (x_0)$ if and only if $r_0$ fails to be semi-convex at $x$. Therefore, $\mathbf{Cut}_0 (x_0)$ has $\mu$ measure 0. Finally, at any point $x$ for which the function $x \to r_0 (x)  $ is differentiable, there exists a unique length minimizing sub-Riemannian geodesic and this geodesic is normal.

\

We now introduce the relevant tensors to state the horizontal Laplacian comparison theorem. We first define  for $X \in \Gamma^\infty(\mathcal{H})$,
\[
\mathbf{K}_{\mathcal{H},J} (X,X) = \langle R( X,\mathbf{J}X)\mathbf{J}X , X\rangle_\mathcal{H} .
\]
The quantity $\mathbf{K}_{\mathcal{H},J}$ is sometimes called the pseudo-Hermitian sectional curvature of the Sasakian manifold (see \cite{B} for a geometric interpretation). It can be seen as the CR analog of the holomorphic sectional curvature of a K\"ahler manifold.

We will  also denote
\[
\mathbf{Ric}_{\mathcal{H},J^\perp} (X,X) =\mathbf{Ric}_\mathcal{H} (X,X)-\mathbf{K}_{\mathcal{H},J} (X,X). 
\]
Recall that for an $n$-dimensional Riemannian manifold with Ricci curvature bounded from below by $(n-1) k$, the usual Laplacian comparison theorem states that $\Delta r \leq (n-1) F_{\mathrm{Rie}}(r,k)$ where
$$F_{\mathrm{Rie}}(r,k) = \begin{cases}  \sqrt{k} \cot \sqrt{k} r & \text{if $k > 0$,} \\
\frac{1}{r} & \text{if $k = 0$,}\\ \sqrt{|k|} \coth \sqrt{|k|} r & \text{if $k < 0$.} \end{cases}$$

Furthermore, for $3$-dimensional Sasakian manifolds with $\mathbf{K}_{\ch,J} \geq k$ on horizontal vectors, we have the mentioned sharp inequality $\Delta_\ch r \leq F_{\mathrm{Sas}}(r,k)$ of \cite{AL1}, where
$$F_{\mathrm{Sas}}(r,k) = \begin{cases}  \frac{\sqrt{k}(\sin \sqrt{k}r  -\sqrt{k} r \cos \sqrt{k} r)}{2 - \cos \sqrt{k} r - \sqrt{k} r \sin \sqrt{k} r} & \text{if $k > 0$,} \\
\frac{4}{r} & \text{if $k = 0$,}\\ \frac{\sqrt{|k|}( \sqrt{|k|} r \cosh \sqrt{|k|} r - \sinh \sqrt{|k|}r)}{2 - \cosh \sqrt{|k|} r + \sqrt{|k|} r \sinh \sqrt{|k|} r} & \text{if $k < 0$.} \end{cases}$$
We generalize this result to arbitrary dimensions in our main result.

\begin{theorem}[Horizontal Laplacian comparison theorem] \label{th:SasakianComp3}
Let $(\M,\mathcal{F},g)$ be a Sasakian foliation with sub-Riemannian distance $d_0$. Define $r_0(x) = d(x_0 ,x)$.
Assume that for some $k_1, k_2 \in \mathbb{R}$
$$\mathbf{K}_{\mathcal{H},J}(v,v) \ge  k_1, \qquad \mathbf{Ric}_{\mathcal{H},J^\perp}(v,v) \ge (n-2)k_2, \qquad v \in \ch, \| v\|_g = 1.$$
Then outside of the $d_0$ cut-locus of $x_0$ and globally on $\M$ in the sense of distributions,
\[
\Delta_\Ho r_0 \le F_{\mathrm{Sas}}(r,k_1) +    (n-2)  F_{\mathrm{Rie}}(r,k_2).
\]
\end{theorem}

It is known that the  holomorphic sectional curvature determines the whole curvature tensor, however there exist explicit examples of manifolds with positive holomorphic sectional curvature without any metric of positive Ricci curvature (see \cite{hichin}). As a consequence it is likely that there exist examples for which $k_1$ and $k_2$ do not have the same sign.

\

Theorem \ref{th:SasakianComp3} will be proved in the next sections. As a by-product of the proof of this theorem, we first point out a straightforward corollary.

\begin{theorem}[Sub-Riemannian Bonnet-Myers theorems] \label{BMsubriem}
\label{th:BM}
Let $(\M,\mathcal{F},g)$ be a Sasakian foliation.
\begin{enumerate}[\rm1.]
\item 
Assume that for some $k_1 > 0$, we have 
\[
\mathbf{K}_{\mathcal{H},J}(v,v) \ge  k_1, \qquad v \in \ch, \| v\|_g = 1.
\]
Then $\M$ is compact, the fundamental group $\pi_1(\M)$ is finite and
\[
\mathbf{diam} \left( \M , d_0 \right) \le \frac{2\pi}{\sqrt{k_1}}.
\]
\item 
Assume $n >2$ and that for some $k_2 > 0$, we have 
\[
\mathbf{Ric}_{\mathcal{H},J^\perp}(v,v) \ge (n-2)k_2, \qquad v \in \ch, \| v\|_g = 1.
\] 
Then $\M$ is compact, the fundamental group $\pi_1(\M)$ is finite and
\[
\mathbf{diam} \left( \M , d_0 \right) \le \frac{\pi}{\sqrt{k_2}}.
\]
\end{enumerate}
\end{theorem}

\begin{remark}
The same Bonnet-Myers type theorems with identical assumptions were obtained in \textup{\cite{ABR}} (see Corollaries 5.6, 5.8) by completely different methods. As observed in \textup{\cite{ABR}}, the diameter upper bounds are sharp in the case of the  Hopf fibration $\mathbb{S}^1 \to \mathbb{S}^{2n+1} \to \mathbb{CP}^n$ (the sub-Riemannian diameter is $\pi$ in that case).
\end{remark}

\subsection{The curvature tensor on Sasakian manifolds} The following lemma will be useful:

\begin{lemma}\label{positive vertical}
 Let $(\M,\mathcal{F},g)$ be a Sasakian foliation. Then, for all $X,Y \in \Gamma^\infty (\M)$,
 \[
 \left\langle R(X,Y)Y,X \right\rangle_{g_\ve} =  \left\langle R(X_\Ho , Y_\Ho ) Y_\Ho , X_\Ho \right\rangle_{\Ho} .
 \]
\end{lemma}

\begin{proof}
Observe first that from the first Bianchi identity, with $\circlearrowright$ denoting the cyclic sum, we have
$$\circlearrowright R(X,Y)Z = \circlearrowright T(T(X,Y), Z) + \circlearrowright (\nabla_X T)(Y,Z) = 0.$$
The fact that $\nabla$ preserves the metric, gives us $\langle R( \, \cdot \, , \, \cdot \,) v, v \rangle_{g_\ve} =0$. Hence, as $\nabla$ also preserves both subbundles $\Ho$ and $\V$, one obtains
\begin{align*}
& \langle R(X,Y) Y, X \rangle_{g_\ve} - \langle R(X_\Ho,Y_\Ho) Y_\Ho, X_\Ho \rangle_\Ho \\
& = \langle R(X_\Ho,Y_\V) Y_\Ho, X_\Ho \rangle_{g_\ve} + \langle R(X_\V,Y_\Ho) Y_\Ho, X_\Ho \rangle_{g_\ve} + \langle R(X_\V,Y_\V) Y_\Ho, X_\Ho \rangle_{g_\ve} \\
& \quad  + \langle R(X_\Ho ,Y_\Ho) Y_\V, X_\V \rangle_{g_\ve} + \langle R(X_\Ho ,Y_\V) Y_\V, X_\V \rangle_{g_\ve} \\
& \quad + \langle R(X_\V,Y_\Ho) Y_\V, X_\V \rangle_{g_\ve} + \langle R(X_\V,Y_\V) Y_\V, X_\V \rangle_{g_\ve} \\
& = \langle \circlearrowright R(X_\Ho,Y_\V) Y_\Ho, X_\Ho \rangle_{g_\ve} - \langle Y_\Ho, \circlearrowright R(X_\V,Y_\Ho) X_\Ho \rangle_{g_\ve} + \langle \circlearrowright R(X_\V,Y_\V) Y_\Ho, X_\Ho \rangle_{g_\ve} \\
& \quad  + \langle \circlearrowright R(X_\Ho ,Y_\Ho) Y_\V, X_\V \rangle_{g_\ve} - \langle  Y_\V, \circlearrowright R(X_\Ho ,Y_\V) X_\V \rangle_{g_\ve} \\
& \quad + \langle \circlearrowright R(X_\V,Y_\Ho) Y_\V, X_\V \rangle_{g_\ve} + \langle R(X_\V,Y_\V) Y_\V, X_\V \rangle_{g_\ve} \\
& = \langle R(X_\V,Y_\V) Y_\V, X_\V \rangle_{g_\ve} =0,
\end{align*}
where in the last equality we used the fact that the leaves are one-dimensional.
\end{proof}

\subsection{Horizontal Hessian comparison theorem}
Throughout this section, we will rely on the following functions. For $r, \mu \in \mathbb{R}$, we define
\begin{equation*}
\phi_\mu(r) = \begin{cases} \frac{\sinh \sqrt{\mu} r}{\sqrt{\mu}} & \text{if $\mu > 0$,} \\
r & \text{if $\mu = 0$,}\\ \frac{\sin \sqrt{|\mu|} r}{\sqrt{|\mu|}} & \text{if $\mu < 0$,} \end{cases} \qquad 
\psi_\mu(r) =  \begin{cases} \frac{\sinh \sqrt{\mu} r- \sqrt{\mu} r}{\mu^{3/2}} & \text{if $\mu > 0$,} \\
\frac{1}{6} r^3 & \text{if $\mu = 0$,}\\ \frac{\sqrt{|\mu|}r - \sin \sqrt{|\mu|} r}{|\mu|^{3/2}} & \text{if $\mu < 0$.} \end{cases} 
\end{equation*}

Notice  that $\psi_\mu (r) = \int_0^r \int_0^{s_2} \phi_\mu(s_1) \, ds_1 \, ds_2$. 
We finally introduce the following function:
\begin{equation*}  \Psi_\mu(r) =  \begin{cases} \frac{1}{\mu^{3/2}} (\sqrt{\mu}- \frac{1}{r} \tanh \sqrt{\mu} r) & \text{if $\mu > 0$,}\\
\frac{1}{3} r^2 & \text{if $\mu = 0$,}\\
\frac{1}{|\mu|^{3/2}} ( \frac{1}{r} \tan \sqrt{|\mu|} r - \sqrt{|\mu|}) & \text{if $\mu < 0$.}
\end{cases}\end{equation*}
Using trigonometric and hyperbolic identities, we can verify that
\begin{equation}\label{FtoPhi}
F_{\mathrm{Rie}}(r,k) = \frac{\phi_{-k}'(r)}{\phi_{-k}(r)}, \qquad F_{\mathrm{Sas}}(r,k) = \frac{\phi_{-k}'}{\phi_{-k}} \frac{\Psi_{-k}(r)}{\Psi_{-k}(r/2)}
\end{equation}

As before, let $x_0\in \M$ be fixed and for $\varepsilon \ge 0$ denote
\[
r_\varepsilon (x) =d_\varepsilon (x_0,x).
\]

If $v \in \mathcal{H}_x$ is a non-zero vector, we define the  space $\mathfrak{L}_J (v)$ to be the subspace of $\mathcal{H}_x$ orthogonal to $v$ and  $\mathbf J v$. Observe that $\dim \mathfrak{L}_J (v)=n-2$ and that $u \in \mathfrak{L}_J (v)$ if and only if it is orthogonal to $u$ and satisfies $T(u,v)=0$.

Theorem \ref{th:SasakianComp3} will be proved as a consequence of the following horizontal Hessian comparison theorem.
 
\begin{theorem}[Horizontal Hessian comparison theorem]\label{Hessian general3}
Let $(\M,\mathcal{F},g)$ be a Sasakian foliation. Let $k_1,k_2 \in \mathbb{R}$ and $\ve >0$. Let $x\neq x_0$ be a point that is not in the $g_\ve$ cut-locus of $x_0$.
\begin{enumerate}[\rm1.]
 \item If $\| \nabla_\Ho r_\ve(x)\|_g^2  > 0$, then
\[
\nabla_\Ho^2 r_\varepsilon (\nabla_\Ho r_\varepsilon (x),\nabla_\Ho r_\varepsilon (x)) \le \frac{ \min (\| \nabla_\Ho r_\ve(x)\|_g^2 , 1-\| \nabla_\Ho r_\ve(x)\|_g^2)}{r_\ve (x)}.
\]

\item Assume that $\| \nabla_\Ho r_\ve(x)\|_g^2  > 0$ and  that for every local vector field $X \in \Gamma^\infty(\mathcal{H})$,   $\| X \|_\Ho=1$, 
\begin{align}\label{H1a}
 \langle R( X, \mathbf J X)\mathbf J X , X\rangle_\mathcal{H} \ge k_1.
\end{align}
Then, 
\[
\frac{1}{\lambda_\ve} \nabla^2 r_\ve(\mathbf J \nabla_\Ho  r_\ve(x), \mathbf J \nabla_\Ho r_\ve(x)) \leq\frac{\phi'_{-\lambda_\ve k_1}(r_\ve)}{\phi_{-\lambda_\ve k_1}(r_\ve)} \frac{\lambda_\ve \Psi_{-\lambda_\ve k_1}(r_\ve) + \ve}{\lambda_\ve \Psi_{-\lambda_\ve k_1}(r_\ve/2) + \ve} 
\]
where $ \lambda_\ve=\| \nabla_\Ho r_\ve(x)\|_g^2$.
\item Assume that $\| \nabla_\Ho r_\ve(x)\|_g^2  > 0$ and  that for all local vector fields $X,Y \in \Gamma^\infty(\mathcal{H})$, $\| X \|_g=\|Y\|_g=1$, and $Y \in \mathfrak{L}_J (X)$,
\begin{align}\label{H2a}
 \langle R( X, Y)Y , X\rangle_\mathcal{H} \ge k_2.
\end{align}
Then, for any horizontal unit vector $v \in \mathfrak{L}_J (\nabla_\Ho r_\ve(x))$,
\[
\nabla^2 r_\ve(v, v) \leq \frac{\phi_{-\lambda_\ve k_2}'(r_\ve)}{\phi_{- \lambda_\ve k_2}(r_\ve)}
\]
where $ \lambda_\ve=\| \nabla_\Ho r_\ve(x)\|_g^2$.
\item If $\nabla_\Ho r_\ve(x) = 0$, then  $\nabla^2 r_\ve (v, v) \le \frac{1}{2\ve} \frac{\phi_{-1/\ve}'(r_\ve/2)}{\phi_{- 1/\ve}(r_\ve/2)} $ for any unit  $v \in \Ho_x$.
\end{enumerate}
\end{theorem}
\begin{remark} \label{rem:lambda0}
 Observe that the set of $x\in \M$ such that $\nabla_\Ho r_\ve(x) = 0$ is a bounded set of measure zero included in the leaf passing through $x_0$.
\end{remark}

\begin{proof}
In the proof,  to simplify notations, we often simply write $r= r_\ve(x)$ and $\lambda = \|\nabla_\Ho r_\ve (x)\|^2_g$. For any vector field $Y$ along a geodesic $\gamma$, we will use $Y'$ for the covariant derivative with respect to the adjoint connection $\hat{\nabla}^\ve_{\gamma'}$ and we identify vectors and their corresponding parallel vector field along $\gamma$.
\begin{enumerate}[1.]
\item From the index lemma, one has 
\[
\nabla_\Ho^2 r_\varepsilon (\nabla_\Ho r_\varepsilon,\nabla_\Ho r_\varepsilon) \le I_{\varepsilon} (\gamma,X,X)
\]
 where $\gamma$ is the unique length parametrized geodesic connecting $x_0$ to $x$ and $X=\frac{t}{r_\varepsilon} \gamma_\Ho'$. An immediate computation gives
 \[
 I_{\varepsilon} (\gamma,X,X)=\frac{\Gamma(r_\varepsilon)}{r_\varepsilon}.
 \]
 Therefore
 \[
 \nabla_\Ho^2 r_\varepsilon (\nabla_\Ho r_\varepsilon,\nabla_\Ho r_\varepsilon) \le \frac{\Gamma(r_\varepsilon)}{r_\varepsilon}.
 \]
 We now observe that
\[
\| \nabla_\Ho r_\varepsilon \|^2_\Ho +\varepsilon \| \nabla_\V r_\varepsilon \|^2_\V =1.
\]
As a consequence
\[
\nabla_\Ho \| \nabla_\Ho r_\varepsilon \|^2_\Ho +\varepsilon \nabla_\Ho \| \nabla_\V r_\varepsilon \|^2_\V =0
\]
and
\[
\nabla_\V \| \nabla_\Ho r_\varepsilon \|^2_\Ho +\varepsilon \nabla_\V \| \nabla_\V r_\varepsilon \|^2_\V =0.
\]
From the first equality we deduce
\[
\langle \nabla_\Ho \| \nabla_\Ho r_\varepsilon \|^2_\Ho,  \nabla_\Ho r_\varepsilon \rangle_\Ho+\varepsilon \langle \nabla_\Ho \| \nabla_\V r_\varepsilon \|^2_\V ,  \nabla_\Ho r_\varepsilon \rangle=0,
\]
and therefore,
\[
 \nabla^2 _\Ho r_\varepsilon ( \nabla_\Ho r_\varepsilon,\nabla_\Ho r_\varepsilon)+\varepsilon  \nabla^2 _{\Ho,\V} r_\varepsilon ( \nabla r_\varepsilon,\nabla r_\varepsilon)=0.
 \]
 Similarly, from the second equality we have
 \[
 \nabla^2 _{\V,\Ho} r_\varepsilon ( \nabla r_\varepsilon,\nabla r_\varepsilon) + \varepsilon \nabla^2 _\V r_\varepsilon ( \nabla_\V r_\varepsilon,\nabla_\V r_\varepsilon)=0.
 \]
 Since the Sasakian foliation is totally geodesic, it is easy to check that $ \nabla^2 _{\V,\Ho}= \nabla^2 _{\Ho,\V}$ (see \cite{BB}). Consequently,
 \[
  \nabla^2 _\Ho r_\varepsilon ( \nabla_\Ho r_\varepsilon,\nabla_\Ho r_\varepsilon)=\varepsilon^2 \nabla^2 _\V r_\varepsilon ( \nabla_\V r_\varepsilon,\nabla_\V r_\varepsilon).
  \]
From the vertical index form in Proposition \ref{index formulas} one has
\[
\nabla^2 _\V r_\varepsilon ( \nabla_\V r_\varepsilon,\nabla_\V r_\varepsilon) \le \frac{ \| \nabla_\V r_\varepsilon \|^2_\V}{\varepsilon r_\varepsilon}. 
\]
This yields
\[
\nabla^2 _\Ho r_\varepsilon ( \nabla_\Ho r_\varepsilon,\nabla_\Ho r_\varepsilon)\le \varepsilon  \frac{ \| \nabla_\V r_\varepsilon \|^2_\V }{ r_\varepsilon}=\frac{1- \| \nabla_\Ho r_\varepsilon \|^2_\Ho}{ r_\varepsilon}=\frac{1-\Gamma(r_\varepsilon)}{ r_\varepsilon}.
\]

\item The proof of 2. is the most difficult. The idea is to use an almost Jacobi field based on the computations of Appendix 2 (to which we refer for further details). Let $\gamma$ be the unit speed geodesic joining $x_0$ and $x$. 
Define 
$$C_\ve = \lambda\left(\psi'_{- \lambda k_1}(r)^2 - \psi_{-\lambda k_1}(r) \psi_{-\lambda k_1}''(r) \right) + \ve r \psi''_{-\lambda k_1}(r),$$
$$G_\ve(t) = \frac{1}{ C_\ve} \left(  \psi_{-\lambda k_1}'(r)  \psi_{- \lambda k_1}(t) + \left( \frac{\ve}{\lambda} r - \psi_{-\lambda k_1}(r)\right) \psi_{-\lambda k_1}'(t)   \right)$$and 
$$Y(t) = \frac{1}{\lambda}G'_\ve(t) J_Z \gamma_\Ho'  +\left( \ve \psi_{- \lambda k_1}'(r) t + G_\ve(t)\right) S,$$
This vector field satisfies
\begin{align*}
 & Y'' - T(\gamma', Y') + \frac{1}{\ve} J_{Y'} \gamma'_\Ho+ \lambda k_1  Y_\Ho - \frac{1}{\ve} J_{T(\gamma', Y)} \gamma'_\Ho=0
\end{align*}
with boundary conditions $Y(0) = 0$ and $Y(r) = \frac{1}{\sqrt{\lambda}} \mathbf{J} \gamma'_\Ho$. 

Computations with the index form and Lemma \ref{positive vertical} give us
\begin{align*}
I_\ve(Y,Y) & = \int_0^r \left\langle Y' , Y' - T(\gamma', Y) + \frac{1}{\ve} J_Y \gamma' - \frac{1}{\ve} J_{\gamma'} Y \right\rangle_{g_\ve} dt \\
& \quad - \int_0^r \left\langle R(\gamma', Y) Y + \frac{1}{\ve} J_{T(\gamma', Y)} Y,  \gamma' \right\rangle_{g_\ve} dt \\
& = \langle Y(r), Y'(r) \rangle_{g_\ve} - \frac{1}{\ve} \int_0^r \langle T(Y, Y'), \gamma' \rangle_{g_\ve} dt \\
& \quad  -  \int_0^r \left\langle Y , Y'' - T(\gamma, Y') + \frac{1}{\ve} J_{Y'} \gamma - \frac{1}{\ve} J_{T(\gamma',Y)} \gamma'  +  \| \gamma_\Ho' \|^2_g k_1 Y \right\rangle dt \\
& \quad - \int_0^r \left( \left\langle R(\gamma', Y) Y ,  \gamma' \right\rangle - k_1  \lambda \|Y_\Ho \|^2_g \right) dt \\
& = \langle Y(r), Y'(r) \rangle_{g_\ve}  - \int_0^r \left( \left\langle R(\gamma', Y) Y ,  \gamma' \right\rangle - k_1  \lambda \|Y_\Ho \|^2_g \right) dt   \\
& \leq  \lambda G''_\ve(r).
\end{align*}
We finally compute
\begin{align*}
 \lambda G_\ve''(r) & =  \frac{1}{C_\ve}\left( \lambda (\psi_{-\lambda k_1}'(r)  \psi_{- \lambda k_1}''(r)  - \psi_{-\lambda k_1}(r) \psi_{-\lambda k_1}'''(r)) + \ve r  \psi_{-\lambda k_1}'''(r) \right) \\
 & = \frac{r \psi_{-\lambda k_1}'''(r)}{C_\ve} ( \lambda \Psi_{-\lambda k_1 }(r) + \ve).
\end{align*}
From the proof of Lemma~\ref{lemma:ExJacobi}~(b), we know that $C_\ve =  r \psi''_{-\lambda k_1} (r) (\lambda  \Psi_{- \lambda k_1}(r/2) +\ve)$, therefore we obtain that
$$\lambda G''_\ve(r) = \frac{\psi'''_{-\lambda k_1}(r)}{\psi''_{-\lambda k_1}(r)} \frac{\lambda \Psi_{-\lambda k_1}(r) + \ve}{\lambda \Psi_{-\lambda k_1}(r/2) + \ve} = \frac{\phi'_{-\lambda k_1}(r)}{\phi_{-\lambda k_1} (r)} \frac{\lambda \Psi_{-\lambda k_1}(r) + \ve}{\lambda\Psi_{-\lambda k_1}(r/2) + \ve}.$$

\item Let $X$ be defined as 
\[
 X(t)= \frac{\phi'_{-\lambda k_2}(t)}{\phi_{-\lambda k_2}(r)} v_0.
\]
Observe that since this vector field solves the equation
$$ 0= X'' - T( \gamma', X') + \frac{1}{\ve} J_{X'}  \gamma'  + k_2 \lambda X - \frac{1}{\ve} J_{T( \gamma',X)}  \gamma' .$$
and satisfies $T(\gamma', X) = 0$, we have
$$I_\ve(X,X) \leq \langle X(r), X'(r) \rangle + \frac{1}{\ve} \int_0^r \langle X(t) , \mathbf{J} X'(t) \rangle_g dt   \leq \frac{\phi_{-k_2 \lambda_\ve}'(r)}{ \phi_{-k_2 \lambda_\ve}(r)} ,$$

\item Define a vector field $X$  by
\begin{align*}  \nonumber
X(t) &= \frac{1}{2 \left(1- \cos \frac{r}{\sqrt{\ve}} \right)} \left( \left( 1+  \cos \frac{r-t}{\sqrt{\ve}} - \cos \frac{r}{\sqrt{\ve}} - \cos \frac{t}{ \sqrt{\ve} }  \right) v_0 \right. \\ 
& \left.  \qquad \qquad \qquad \qquad - \left( \sin \frac{r-t}{\sqrt{\ve}} - \sin\frac{r}{\sqrt{\ve}} + \sin \frac{t}{ \sqrt{\ve} } \right) J_{\dot \gamma} v_0 \right) .
\end{align*} 
By computations similar to Lemma~\ref{lemma:ExJacobi}~(c), $X(t)$ is a Jacobi field. Hence
$$I_\ve(X,X) = \langle X(r), X'(r) \rangle = \frac{ \sin \frac{r}{ \sqrt{\ve} }}{2 \sqrt{\ve} \left(1- \cos \frac{r}{\sqrt{\ve}} \right)} = \frac{1}{2\sqrt{\ve}}  \cot \frac{r}{2 \sqrt{\ve} }  = \frac{1}{2\ve} \frac{\phi'_{1/\ve}(r/2)}{\phi_{1/\ve}(r/2)} .$$
\end{enumerate}
\end{proof}

\subsection{Horizontal Laplacian comparison theorem}

We now turn to the proof of Theorem \ref{th:SasakianComp3}. The first part of the theorem is a straightforward application of Theorem \ref{Hessian general3} and standard arguments, choosing an orthonormal basis at $x$:
\[
\left\{ \frac{1}{ \|\nabla_\Ho r_\ve\|_g}\nabla_\Ho r_\ve ,\frac{1}{ \|\nabla_\Ho r_\ve\|_g} \mathbf J  \nabla_\Ho r_\ve , v_1,\ldots, v_{n-2}\right\}
\]
with $v_1,\ldots,v_{n-2}$ orthonormal basis of the space $\mathfrak{L}_J (\nabla_\Ho r_\ve(x))$. From this result, we obtain the following statement.

\begin{theorem}[Horizontal Laplacian comparison theorem for $d_\ve$]
Let $(\M,\mathcal{F},g)$ be a Sasakian foliation. 
Let $k_1,k_2 \in \mathbb{R}$ and $\ve >0$. Assume that for every $X \in \Gamma^\infty(\mathcal{H})$, $\| X \|_g=1$,
$$\mathbf{K}_{\mathcal{H},J}(v,v) \ge  k_1, \qquad \mathbf{Ric}_{\mathcal{H},J^\perp}(v,v) \ge (n-2)k_2, \qquad v \in \ch, \| v\|_g = 1.$$
 Let $x \neq x_0$ which is not in the cut-locus of $x_0$. Define $\lambda_\ve(x) = \| \nabla_\Ho r_\ve(x)\|^2$ and assume $\lambda_\ve(x)>0$.
Then at $x$ we have
$$\Delta_\Ho r_\ve \le \frac{1}{r_\ve} \min \left\{ 1, \frac{1}{\lambda_\ve} -1\right\}  +   (n-2)  \frac{\phi_{-\lambda_\ve k_2}'(r_\ve)}{\phi_{-\lambda_\ve k_2}(r_\ve)} +  \frac{\phi'_{-\lambda_\ve k_1}(r_\ve)}{\phi_{-\lambda_\ve k_1}(r_\ve)} \frac{\lambda_\ve \Psi_{-\lambda_\ve k_1}(r_\ve) + \ve}{\lambda_\ve \Psi_{-\lambda_\ve k_1}(r_\ve/2) + \ve} .$$
\end{theorem}

In order to obtain Theorem~\ref{th:SasakianComp3}, we need to prove that we can take the limit as $\ve \to 0$. Since the cut-locus of  $x_0$ for the metric $g_\varepsilon$ has measure zero, by usual arguments (Calabi's trick), we have in the sense of distributions:
 \[
\Delta_\Ho r_\varepsilon \le \frac{1}{r_\ve} \min \left\{ 1, \frac{1}{\lambda_\ve} -1\right\}  +   (n-2)  \frac{\phi_{-\lambda_\ve k_2}'(r_\ve)}{\phi_{-\lambda_\ve k_2}(r_\ve)} +  \frac{\phi'_{-\lambda_\ve k_1}(r_\ve)}{\phi_{-\lambda_\ve k_1}(r_\ve)} \frac{\lambda_\ve \Psi_{-\lambda_\ve k_1}(r_\ve) + \ve}{\lambda_\ve \Psi_{-\lambda_\ve k_1}(r_\ve/2) + \ve} .
\]

Indeed, from Calabi's lemma, one has $\M=\mathbf{Cut}_{\ve} (x_0) \cup \Omega$ where $\Omega$ is a star-shaped domain. Take now a family of smooth star-shaped domains $\Omega_n \subset \Omega$, with $\lim \Omega_n =\Omega$ obtained by shrinking $\Omega$ in the $r_\ve$ direction. Consider now a function $f \in C_0^\infty(\M)$ which is non-negative. One
has
\[
\int_\M r_\ve \Delta_{\Ho} f d\mu =-\int_\M \langle \nabla_\Ho f, \nabla_\Ho r_\ve \rangle d\mu=-\lim_{n \to \infty} \int_{\Omega_n} \langle \nabla_\Ho f, \nabla_\Ho r_\ve \rangle d\mu
\]
where we used in the last equality $\| \nabla_\Ho r_\ve \| \le 1$ and $\nabla_\Ho f$ bounded. Similarly,
\[
\int_\M r_\ve \Delta_{\V} f d\mu =-\int_\M \langle \nabla_\V f, \nabla_\V r_\ve \rangle d\mu=-\lim_{n \to \infty} \int_{\Omega_n} \langle \nabla_\V f, \nabla_\V r_\ve \rangle d\mu.
\]
From Green's formula, we have
\[
-\int_{\Omega_n} \langle \nabla_\Ho f, \nabla_\Ho r_\ve \rangle d\mu \le \int_{\Omega_n} (\Delta_{\Ho} r_\ve ) f d\mu+\int_{\Omega_n} \langle \nabla_\V f, \nabla_\V r_\ve \rangle_{g_\ve} d\mu+\ve \int_{\Omega_n} (\Delta_{\V} r_\ve ) f d\mu.
\]
When $n \to \infty$, we have $\int_{\Omega_n} \langle \nabla_\V f, \nabla_\V r_\ve \rangle_{g_\ve} d\mu+\ve \int_{\Omega_n} (\Delta_{\V} r_\ve ) f d\mu \to 0$.
This means that for every smooth, non-negative  and compactly supported function $f$,
\begin{align*}
 & \int_\M (\Delta_\Ho f)  \ r_\varepsilon d\mu \\
&\le  \int_\M \left(\frac{1}{r_\ve} \min \left\{ 1, \frac{1}{\lambda_\ve} -1\right\}  +   (n-2)  \frac{\phi_{-\lambda_\ve k_2}'(r_\ve)}{\phi_{-\lambda_\ve k_2}(r_\ve)} +  \frac{\phi'_{-\lambda_\ve k_1}(r_\ve)}{\phi_{-\lambda_\ve k_1}(r_\ve)} \frac{\lambda_\ve \Psi_{-\lambda_\ve k_1}(r_\ve) + \ve}{\lambda_\ve \Psi_{-\lambda_\ve k_1}(r_\ve/2) + \ve}  \right) f d \mu.
\end{align*}
Taking the limit as $\varepsilon \to 0$ yields the result, thanks to Lemma \ref{limit} and equations \eqref{FtoPhi}.

\subsection{Proof of Theorem~\ref{BMsubriem}}
The proof is relatively similar to the proof of Corollary \ref{BMyers}. We will only prove $k_1>0$, since the proof of $k_2 >0$ is almost identical. Let $\ve > 0$. Since
\[
\lim_{r_\ve \to \frac{2\pi}{\sqrt{\lambda_\ve k_1}}} \frac{\phi'_{-\lambda_\ve k_1}(r_\ve)}{\phi_{-\lambda_\ve k_1}(r_\ve)} \frac{\lambda_\ve \Psi_{-\lambda_\ve k_1}(r_\ve) + \ve}{\lambda_\ve \Psi_{-\lambda_\ve k_1}(r_\ve/2) + \ve}=-\infty,
\]
one deduces from Theorem~\ref{Hessian general3} that if $x$ is not in the cut-locus of $x_0$, then $d_\ve (x_0,x) \le \frac{2\pi}{\| \nabla_\Ho r_\ve(x)\| \sqrt{k_1}}$. We conclude from Lemma \ref{limit} that for almost every $x$, we have $d_0 (x_0,x) \le \frac{2\pi}{ \sqrt{k_1}}$.

To complete the proof, we note first that since the foliation is Riemannian, for every sufficiently small neighborhood $U$ in $\M$ such that $\pi_U: U \to \M_U = U/\mathcal{F}|U$ is smooth map of manifolds, there exist a Riemannian metric $g_U$ on $\M_U$ such that $g_\ch = \pi^* g_U$. Furthermore, if $R^U$ denotes the curvature of the Levi-Civita connection of $g_U$ and $R$ is the curvature of the Bott connection, then for any vector fields $X$ and $Y$ on $U$, we have
$$ \left\langle R(X_\Ho , Y_\Ho ) Y_\Ho , X_\Ho \right\rangle_{\Ho}  =  \left\langle R^U( \pi_{U,*} X , \pi_{U,*} Y ) \pi_{U,*}Y , \pi_{U,*} X \right\rangle_{g_U} $$
See \cite[Section~3.1]{GT1} for details. In conclusion, $k_1$ only depends on the Riemannian geometry of $\M_U$ for all sufficiently small neighborhoods $U$ of $\M$. Next. let $p: \tilde \M\to \M$ denote the universal cover of $\M$. Consider the foliation and metric $(\tilde \M, \tilde{ \mathcal{F}}, \tilde g)$ obtained by pulling these back from $\M$. The foliation $\mathcal{F}$ is then Riemannian with totally geodesic leaves since the equations of \eqref{RTG} only depend on local properties. The same is true for the requirements for the foliation to be Sasakian, so if we can show that its pseudo-Hermitian curvature $\tilde{ \mathbf{K}}$ will be bounded from below by $k_1$. However, this is true, since for every sufficiently small neighborhood $\tilde U$ such that $p :\tilde U \to U = p(U)$ is an isometry and such that $\tilde U / \tilde {\mathcal{F}}|\tilde U$ is a manifold, we have that $\tilde U / \tilde {\mathcal{F}}|\tilde U$ is isometric to $\M_U$ as well. The result follows.

\subsection{Vertical Hessian and Laplacian comparison theorems}

One can also easily prove vertical Hessian and Laplacian comparison theorems.

\begin{theorem}[Vertical Hessian comparison theorem]\label{Hessian general4}
Let $(\M,\mathcal{F},g)$ be a Sasakian foliation.  Let $k_1 \in \mathbb{R}$ and $\ve >0$. Let $x\neq x_0$ be a point that is not in the $g_\ve$ cut-locus of $x_0$. Assume that $\| \nabla^\Ho r_\ve(x)\|_g^2  > 0$ and  that for every $X \in \Gamma^\infty(\mathcal{H})$,  $\| X \|_\Ho=1$, 
\begin{align*}
\mathbf{K}_{\mathcal{H},J} (X,X) \ge  k_1.
\end{align*}
Then, for any $g$-unit vertical vector $z \in \mathcal{V}_{x}$,
\[
 \nabla^2 r_\ve(z,z) \leq\frac{ \phi_{-\lambda_\ve k_1}(r_\ve)}{\phi_{-\lambda_\ve k_1}(r_\ve) (\ve r_\ve -\psi_{-\lambda_\ve k_1} (r_\ve)) +\psi'_{-\lambda_\ve k_1} (r_\ve)^2 }  
\]
where $ \lambda_\ve=\| \nabla_\Ho r_\ve(x)\|_g^2$.
\end{theorem}
Note that a simple computation shows 
$\psi'_{- \lambda_\ve k_1} ( r_\ve )^2 - \phi_{- \lambda_\ve k_1} ( r_\ve ) \psi_{- \lambda_\ve k_1} ( r_\ve ) > 0$ if $r_\ve > 0$ for $\ve \ge 0$. 
Actually, when $k_1 >0$, 
\begin{multline*}
\psi'_{-\lambda_\ve k_1} ( r_\ve )^2 - \phi_{-\lambda_\ve k_1} ( r_\ve ) \psi_{-\lambda_\ve k_1} ( r_\ve )
\\
=
\frac{ 4 }{ \lambda_\ve^2 k_1^2 }
\sin \left( \frac{ \sqrt{ \lambda_\ve k_1 } r_\ve }{ 2 } \right) 
\left( 
\sin \left( \frac{ \sqrt{ \lambda_\ve k_1 } r_\ve }{ 2 } \right) 
-
\frac{ \sqrt{ \lambda_\ve k_1 } r_\ve }{ 2 } \cos \left( \frac{ \sqrt{ \lambda_\ve k_1 } r_\ve }{ 2 } \right)
\right)
> 0
\end{multline*}
since
$r_\ve \le r_0 \le 2\pi / \sqrt{ k_1} \le 2 \pi / \sqrt{ \lambda_\ve k_1 }$ 
by Theorem~\ref{BMsubriem}. Other cases can be discussed similarly. 
\begin{proof}
The proof is somewhat similar to the proof of Theorem \ref{Hessian general3} (2), so we omit the details. The idea is to consider the vector field defined along a geodesic $\gamma$ by
\[
X=\left( C_1 \phi_{-\lambda_\ve k_1}(t) -\frac{C_0}{\ve}\psi'_{-\lambda_\ve k_1} (t) \right) J_Z \gamma'+\left(C_0\left( t-\frac{1}{\ve} \psi_{-\lambda_\ve k_1} (t)\right) +C_1 \psi'_{-\lambda_\ve k_1} (t)\right)Z,
\]
where $Z$ is parallel transport of $z$ along $\gamma$ for the adjoint connection $\hat{\nabla}^\varepsilon =\nabla+\frac{1}{\varepsilon} J $ and $C_0,C_1$ are the constants such that $X(r_\ve)=z$.
\end{proof}
As an immediate corollary, we deduce:

\begin{corollary}[Vertical Laplacian comparison theorem]\label{VL}
Let $(\M,\mathcal{F},g)$ be a Sasakian foliation. Let $k_1 \in \mathbb{R}$ and $\ve >0$. Assume that for every $X \in \Gamma^\infty(\mathcal{H})$, $\| X \|_g=1$,
\begin{align*}
\mathbf{K}_{\mathcal{H},J} (X,X) \ge  k_1.
\end{align*}
Let $x \neq x_0$ which is not in the cut-locus of $x_0$. Define $\lambda_\ve(x) = \| \nabla_\Ho r_\ve(x)\|^2$ and assume $\lambda_\ve(x)>0$. Then at $x$ we have
\[
\Delta_\V r_\ve \le  \frac{  \phi_{-\lambda_\ve k_1}(r_\ve)}{\phi_{-\lambda_\ve k_1}(r_\ve) (\ve r_\ve -\psi_{-\lambda_\ve k_1} (r_\ve)) +\psi'_{-\lambda_\ve k_1} (r_\ve)^2 }  .
\]
Therefore, outside of the $d_0$ cut-locus of $x_0$ and globally on $\M$ in the sense of distributions,
\[
\Delta_\V r_0 \le \frac{ \phi_{- k_1}(r_0)}{-\phi_{- k_1}(r_0)\psi_{- k_1} (r_0) +\psi'_{-k_1} (r_0)^2 }  .
\]
\end{corollary}

\begin{remark}
When $k_1=0$, the theorem yields $\Delta_\V r_0 \le \frac{12}{r^3_0}$. 
\end{remark}
\subsection{Measure contraction properties}

As an application of Theorem~\ref{th:SasakianComp3}, 
we will show  measure contraction properties of the metric measure spaces $(\M,d_\ve,\mu)$, $\varepsilon \ge 0$ (see  \cite{Ohta,St1,St2} for standard corollaries of the measure contraction properties). 
To state it, we prepare some notations. Let $\ve \ge 0$.
Let $e_t\colon C ( [ 0, 1 ]; \M ) \to \M$ be the evaluation map 
for $t \in [ 0 , 1 ]$ given by $e_t ( \gamma ) = \gamma_t$. 
For a probability measure $\nu$ on $\M$ and $x_0 \in \M$, 
there exists a probability measure $\Pi$ on the space of 
(constant speed) minimal geodesics $\mathrm{Geo}_\ve (\M)$ 
on $( \M , g_\ve )$ 
satisfying $( e_0 )_\sharp\Pi = \delta_{x_0}$ and $( e_1 )_\sharp \Pi = \nu$. 
Such a $\Pi$ is called a dynamic optimal coupling from $\delta_{x_0}$ to $\nu$. 
In our case, 
we have a measurable map $G_\ve : \M \to \mathrm{Geo}_\ve (\M)$ so that 
each $G_\ve (x)$ is a minimal $g_\ve$ geodesic from $x_0$ to $x$ 
by a measurable selection theorem (the existence of such map is classical when $\ve >0$ and we refer to \cite{LLZ} in the case $\ve=0$).   
Then, the push-forward measure $G_\sharp \nu$ indeed provides a dynamic optimal coupling 
from $\delta_{x_0}$ to $\nu$. 
For $\gamma \in \mathrm{Geo}_\ve (\M)$, 
we denote the $g_\ve$-length of $\gamma$ by $\ell (\gamma)$
(we omit $\ve$ for simplicity of notations). Let $\mu$ denote the volume measure of $g$.
For $A \in \mathcal{B} (\M)$ (Borel set in $\M$) with $\mu (A) \in ( 0, \infty )$, 
let $\bar{\mu}_A$ be a probability measure on $\M$ given by 
the normalization of the restriction of $\mu$ on $A$: 
$$\mu : = \mu (A)^{-1} \mu |_A.$$ 
Again, we write $\lambda_\ve: \mathbb{M} \to \mathbb{R}$ for the function $\lambda_\ve(x) = \| \nabla_\Ho r_\ve\|_\Ho(x)^2$. By slight abuse of notation, let us also define $\lambda_\ve: \mathrm{Geo}_\ve(\mathbb{M}) \to \mathbb{R}$ such that for any constant speed geodesic $\gamma \in  \mathrm{Geo}_\ve(\mathbb{M})$ starting at $x_0$, $\lambda_\ve(\gamma)=\frac{\| \gamma_\Ho'(t)\|^2_\Ho}{l(\gamma)^2}$, $t \in [0,1]$ which is a constant by \eqref{observation}. 
Additionally, let us define a function $\Phi_{\ve, , \lambda, \kappa}$ 
and $\Xi_{\ve, \kappa}$ for $\ve > 0$, $\lambda \in (0,1]$ and $\kappa \in \R$ by 
\begin{align*}
\Phi_{\ve, \lambda, \kappa} (r) : = & 
\begin{cases}
\ds
  \lambda
  ( 2 \kappa^{-1} ( 1 - \phi'_{-\kappa} (r) ) - r \phi_{-\kappa} (r) ) 
  + \ve 
  \phi_{-\kappa} (r),
& \kappa \neq 0, 
\\
\ds 
r ( \lambda r^2 + 12 \ve )^{3/2},
& \kappa = 0 , 
\end{cases}
\\
\Xi_{\ve, \kappa} (r) : = & 
 \frac{  \phi_{- \kappa}(r)}{\phi_{- \kappa}(r) (\ve r -\psi_{- \kappa} (r)) +\psi'_{- \kappa} (r)^2 } .
\end{align*} 
We also write $\Phi_{\kappa} : = \Phi_{0,1,\kappa}$.  

\begin{theorem}[Measure contraction property] 
\label{th:MCP}
Let $(\M,\mathcal{F},g)$ be a Sasakian foliation.
Assume that for constants $k_1, k_2 \in \mathbb{R}$ and for every $X \in \Gamma^\infty(\mathcal{H})$ with $\| X \|_g=1$,
\begin{align}\label{H1bis_again}
\mathbf{K}_{\mathcal{H},J} (X,X) \ge  k_1,
\end{align}
and,
\begin{align}\label{H2bis_again}
\mathbf{Ric}_{\mathcal{H},J^\perp} (X,X) \ge (n-2)k_2 .
\end{align}
\begin{enumerate}[\rm(1)]
\item
For any $\ve > 0$, $A \in \mathcal{B} (\M)$ with $\mu (A) \in ( 0 , \infty )$ and $x_0 \in \M$, 
there exists a dynamic optimal coupling $\Pi$ on the space of 
(constant speed) minimal geodesics $\mathrm{Geo}_\ve (\M)$ from $\delta_{x_0}$ to $\bar{\mu}_A$ 
such that the following holds: 
\begin{multline}\label{eq:pre-MCP}
\mu \ge 
( e_t )_\sharp 
\Bigg( 
  \frac{
    t^{1+\min \{ 1, \lambda_\ve^{-1} - 1 \} } 
    \phi_{-\lambda_\ve k_2}^{n-2} (t \ell (\gamma))
    \Phi_{\ve, \lambda_\ve, \lambda_\ve k_1} (t \ell (\gamma)) 
  }
  { 
    \phi_{-\lambda_\ve k_2}^{n-2} (\ell (\gamma))
    \Phi_{\ve,\lambda_\ve, \lambda_\ve k_1} (\ell (\gamma)) 
  } 
\\
\times 
  \exp \left( 
      \ve \int_{\ell (\gamma)}^{t \ell (\gamma)} \Xi_{\ve , \lambda_{\ve} k_1} (r) dr 
  \right)
  \mu (A)
\Pi \Bigg). 
\end{multline}
\item
For any $A \in \mathcal{B} (\M)$ with $\mu (A) \in ( 0 , \infty )$ and $x_0 \in \M$,  
there exists a dynamic optimal coupling $\Pi$ from $\delta_{x_0}$ to $\bar{\mu}_A$ on the space of 
(constant speed) minimal geodesics $\mathrm{Geo}_0 (\M)$
such that the following holds: 
\begin{equation*} 
\mu \ge 
( e_t )_\sharp 
\left( 
  \frac{
    t
    \phi_{- k_2}^{n-2} (t \ell (\gamma))
    \Phi_{k_1} (t \ell (\gamma)) 
  }
  { 
   \phi_{- k_2}^{n-2} (\ell (\gamma))
    \Phi_{k_1} (\ell (\gamma)) 
  } 
  \mu (A)
\Pi \right). 
\end{equation*}

\end{enumerate}
\end{theorem}

\begin{remark}
Theorem~\ref{th:MCP} (2) asserts the same inequality 
as in \cite[Theorem 1.1]{LLZ}. 
As we will see below, our approach gives 
an alternative simple proof of this result. 
\end{remark}

A strong connection between Laplacian comparison theorems and  
measure contraction properties in an infinitesimal form 
are known 
(see \cite{Lee}, \cite[Section 6.2]{Ohta_H} for instance; 
cf.\ \eqref{eq:v-distort} below). 
Here we will give a detailed proof for completeness. 
One reason why we prefer it is on the fact that Laplacian comparison theorem 
is described in terms of Laplacian and distance 
while measure contraction property is formulated in terms of 
distance and measure. Laplacian, distance and measure are 
mutually related in Riemannian geometry but the same relation is 
not obvious (even not always true) in sub-Riemannian setting. 
Another reason is on the fact we are formulating the measure contraction 
property in an integrated form. Thus the presence of cut locus should 
be treated somehow. It can be problematic when $\ve = 0$. 
Thus we first show the measure contraction property when $\ve > 0$ 
and let $\ve \to 0$ instead of showing it directly from the Laplacian 
comparison theorem when $\ve = 0$. 

\begin{proof}
(1) In the case $\varepsilon > 0$, 
we closely follow the argument in \cite[Section 3]{Ohta}. 
For $x_0 \in \mathbb{M}$, 
let $\mathsf{D}_\ve (x_0) \subset T_{x_0} \mathbb{M}$ be 
the maximal domain of the $g_\ve$-exponential map 
$\exp_{\ve,x_0}$ at $x_0$. 
That is, 
$\mathbb{M} \setminus \exp_{\ve,x_0} ( \mathsf{D}_\ve (x_0) )$ 
is the $g_\ve$ cut-locus $\mathbf{Cut}_{\ve}(x_0)$ of $x_0$. 
Let $\mu_\ve$ be the Riemannian measure for $g_\ve$. 
By definition, we can easily see $\mu_\ve = \ve^{-1/2} \mu$. 
Thus it suffices to show the assertion for $\mu_\ve$ instead of $\mu$, 
since our goal \eqref{eq:pre-MCP} is linear in $\mu$. 
We denote the density of $( \exp_{\ve,x_0}^{-1} )_{\sharp} \mu_\ve$ 
in polar coordinate $(r, \xi)$ 
($r > 0$, $\xi \in T_{x_0} \mathbb{M}$, $|\xi| = 1$)
on $\mathsf{D} (x_0)$ by $A_{\ve,x_0} ( r , \xi )$. 
Then we know 
\begin{equation} \label{eq:v-distort}
\frac{\partial}{\partial r} A_{\ve,x_0} ( r , \xi ) 
= 
\Delta_\ve r_\ve (\exp_{\ve,x_0} (r,\xi))\cdot A_{\ve, x_0} ( r, \xi ), 
\end{equation}
where $\Delta_\ve$ is the Laplace-Beltrami operator for $g_\ve$ 
(see \cite[Theorem 3.8]{Chavel} or \cite[Section 9.1]{Petersen}). 
Since $\Delta_\ve = \Dh + \ve \Dv$, by Theorem~\ref{th:SasakianComp3} 
and Corollary~\ref{VL} (see the comment after Theorem~\ref {Hessian general4} also) together with \eqref{FtoPhi} and 
a simple computation, 
\begin{align*}
\Delta_\ve r_\ve 
& \le 
\frac{1}{r_\ve} \min \left\{ 1, \frac{1}{\lambda_\ve} -1\right\}  
+ (n-2) 
\frac{\phi_{- \lambda_\ve k_2}'(r_\ve)}{\phi_{- \lambda_\ve k_2}(r_\ve)} 
+ 
\frac{\phi'_{-\lambda_\ve k_1}(r_\ve)}{\phi_{-\lambda_\ve k_1}(r_\ve)} 
\frac{\lambda_\ve \Psi_{-\lambda_\ve k_1}(r_\ve) + \ve}
{\lambda_\ve \Psi_{-\lambda_\ve k_1}(r_\ve/2) + \ve} 
\\
& \hspace{30em} 
+ \ve \Xi_{\ve , \lambda_\ve k_1} (r_\ve) 
\\
& = 
\frac{1}{r_\ve} \min \left\{ 1, \frac{1}{\lambda_\ve} -1\right\}  
+ (n-2) 
\frac{\phi'_{-\lambda_\ve k_2} (r_\ve)}{\phi_{-\lambda_\ve k_2}(r_\ve)} 
+ 
 \frac{ \Phi_{\ve, \lambda_\ve, \lambda_\ve k_1}' ( r_\ve ) }
{ \Phi_{\ve, \lambda_\ve, \lambda_\ve k_1} ( r_\ve ) }
+ \ve \Xi_{\ve , \lambda_\ve k_1} (r_\ve) 
\end{align*}
When $\lambda_\ve > 0$. 
Recall that, as observed after \eqref{observation}, 
$\lambda_\ve = \lambda_\ve (\exp_{\ve,x_0} ( r, \xi ))$ does not 
depend on $r$. 
Thus, we regard it constant when we fix $\xi$. 
Then by integrating \eqref{eq:v-distort} in $r$ with applying this inequality, 
for $0 < r_1 < r_2$ with $(r_2 , \xi ) \in \mathsf{D}(x_0)$, we obtain 
\begin{align} \label{eq:pre-MCP0}
\frac{ A_{\ve, x_0} ( r_2 , \xi ) }{ A_{\ve , x_0} ( r_1 , \xi ) }
\le &
\frac{\Theta( \lambda_\ve , r_2 )}{\Theta( \lambda_\ve , r_1 )} ,
\end{align}
where
\begin{align*}
\Theta ( \lambda, r ) 
: = & 
r^{\min \{ 1, \lambda^{-1} - 1 \} } 
  \phi_{-\lambda k_2}^{n-2} (r)
  \Phi_{\ve,\lambda, \lambda k_1} (r) 
  \exp \left( \ve \int_c^r \Xi_{\ve, \lambda k_1} (s) ds \right) 
\end{align*}
for some $c > 0$. 
Let $f \in C^\infty_0 (\M)$ supported on $\exp_{\ve, x_0} ( t \overline{\mathsf{D} (x_0)} )$ 
with $f \ge 0$. 
It suffices to show the integral of $f$ by $\mu$ is 
larger than the integral of $f$ by the right hand side of \eqref{eq:pre-MCP}, 
since the measure on the right hand side of \eqref{eq:pre-MCP} 
is supported on $\exp_{\ve,x_0} ( t \overline{\mathsf{D} (x_0)} )$ by definition. 
Let $G_\ve : \M \to \mathrm{Geo}_\ve (\M)$ be the map mentioned at the beginning 
of this subsection for $\nu = \bar{\mu}_A$. 
Since 
$
( ( \exp_{\ve,x_0}^{-1} )_\sharp \mu ) 
( t \overline{\mathsf{D}(x_0)} \setminus t \mathsf{D} (x_0) ) 
= 0
$, 
with keeping Remark~\ref{rem:lambda0} in mind, we have 
\begin{align*}
\int_{\M} f \, d \mu 
& = 
\int_{t \mathsf{D}(x_0)}
  f ( \exp_{\ve, x_0} ( r , \xi ) )
  A_{\ve, x_0} ( r , \xi ) 
\, d r d \xi
\\
& = 
\int_{\mathsf{D}(x_0)}
  f ( \exp_{\ve, x_0} ( t r , \xi ) )
  t A_{\ve, x_0} ( t r , \xi ) 
\, d r d \xi
\\
& \ge 
\int_{\mathsf{D}(x_0)}
  f ( \exp_{\ve, x_0} ( t r , \xi ) )
\frac{\Theta (\lambda_\ve , tr ) }{ \Theta (\lambda_\ve , r )} 
A_{\ve , x_0} ( r , \xi )\, d r d \xi 
\\
& = 
\int_{\mathsf{D}(x_0)}
  f ( e_t ( G_\ve ( \exp_{\ve, x_0} (r, \xi) ) ) ) 
\frac{\Theta (\lambda_\ve , tr ) }{ \Theta (\lambda_\ve , r )} 
A_{\ve , x_0} ( r , \xi )\, d r d \xi 
\\
& = 
\int_{\mathrm{Geo}_\ve(\M)}
  f ( e_t (\gamma) ) \mu ( A )
\frac{
  \Theta (\lambda_\ve , t \ell (\gamma) ) 
}
{
  \Theta (\lambda_\ve , \ell (\gamma) )
} 
\, ( G_\ve )_\sharp \bar{\mu}_A ( d \gamma ). 
\end{align*}
Here the inequality follows from \eqref{eq:pre-MCP0}, and 
we have used $\mu (\mathbf{Cut}_\ve (p) ) = 0$ in the last identity.%
\medskip

\noindent
(2) Subdividing $A$ by taking an intersection with annuli 
(with respect to $d_0$), 
we may assume that $A$ is bounded. 
Then our claim may be studied only in a (closed) 
$d_0$-ball of sufficiently large radius. 
Let $f \in C_0^\infty (\M)$ with $f \ge 0$. 
Following a naive idea, 
we apply \eqref{eq:pre-MCP} to integrations of $f$ 
and let $\ve \downarrow 0$ with the Fatou lemma. 
Indeed, the density of the right hand side of 
\eqref{eq:pre-MCP0} is non-negative.
By the proof of (1), we may assume also that 
$\Pi$ in \eqref{eq:pre-MCP} is 
of the form $( G_\ve )_\sharp \bar{\mu}_A$. 
By Lemma~\ref{limit}, 
we have $\lim_{\ve \downarrow 0} \lambda_\ve (x) = 1$ $\mu$-a.e. This implies that  $\lim_{\ve \to 0} \lambda_{\ve}( G_\ve(x))=1$.
By (1), it is sufficient to take the limit of $\lambda_\ve$ 
in Fatou's lemma. 
Note that $\ell ( G_\ve (x) ) = d_\ve ( x_0 , x )$ and hence 
$\ell ( G_\ve (x) ) \to d_0 ( x_0, x ) = \ell ( G_0 (x) )$ as $\ve \downarrow 0$. 
Thus the conclusion follows 
once we have $e_t (G_\ep (x)) \to e_t ( G_0 (x) )$ for $\mu$-a.e.~$x$. 
Suppose $x \notin \mathbf{Cut}_0 (x_0)$. Let us take a decreasing sequence 
$( \ve_n )_{ n \in \mathbb{N}}$ with $\ve_n \to 0$. 
Since $d_\ve$ is non-increasing in $\ve$,
\[
d_{\ep_1}  ( x , e_t ( G_{\ve_n} (x) ) ) 
\le 
d_{\ep_n} ( x , e_t ( G_{\ve_n} (x) ) ) 
= 
( 1 - t ) d_{\ep_n} ( x_0 , x ). 
\]
Since the right hand side converges to $( 1 - t ) d_0 ( x_0 , x )$, 
$( e_t ( G_{\ve_n} (x) ) )_{n \in \mathbb{N}}$ is a $d_{\ve_1}$-bounded sequence. 
Thus there is a subsequence $( \ve_{n(k)} )_{k \in \mathbb{N}}$ such that 
$\lim_{k \to \infty} e_t ( G_{\ve_{n(k)}} (x) )$ exists. 
We denote the limit by $y$. 
Then, for $k < k'$, 
\begin{align*}
d_{\ve_{n(k)}} ( x_0 , e_t ( G_{\ve_{n(k')}} (x) ) ) 
& \le
d_{\ve_{n(k')}} ( x_0, e_t ( G_{\ve_{n(k')}} (x) ) ) 
= 
t d_{\ve_{n(k')}} ( x_0 , x ), 
\\ 
d_{\ve_{n(k)}} ( e_t ( G_{\ve_{n(k')}} (x) ) , x )
& \le
d_{\ve_{n(k')}} (  e_t ( G_{\ve_{n(k')}} (x) ) , x )
= 
( 1 - t ) d_{\ve_{n(k')}} ( x_0 , x ). 
\end{align*}
By letting $k' \to \infty$ and $k \to \infty$, we have 
\[
d_0 ( x_0, y )
\le 
t d_0 ( x_0, x ) , \   
d_0 ( y, x )
\le 
( 1 - t ) d_0 ( x_0, x ).
\]
By the triangle inequality $d_0 ( x_0, x ) \le d_0 ( x_0 , y ) + d_0 ( y, x )$, 
both of the last two inequalities must be equalities. 
Since $x \notin \mathbf{Cut}_0 (x_0)$, 
$G_0 (x)$ is a unique minimal geodesic from $x_0$ to $x$. 
Hence we have $y = e_t ( G_0 (x) )$. Thus the claim holds 
since the limit $y$ is independent of the choice of a subsequence 
and $\mathbf{Cut}_0 (x_0)$ is of $\mu$-measure zero. 
\end{proof}

We now recall the following definition (see \cite{Ohta,St1,St2}).

\begin{definition}
Let $(X,\delta,\nu)$ be a metric measure space. Assume that for every $x_0 \in X$ there exists a Borel set $\Omega_{x_0}$ of full measure in $X$ (that is $\nu(X \setminus \Omega_{x_0})=0$) such that any point of $\Omega_{x_0}$ is connected to $x_0$ by a unique distance minimizing geodesic  $t \to \phi_{t,x_0} (x)$, $t \in [0,1]$, starting at $x$ and ending at $x_0$. We say that $(X,\delta,\nu)$  satisfies the measure contraction property $\mathbf{MCP}(0,N)$, $N\ge 0$, if for every $x_0 \in X$, $t \in [0,1]$ and Borel set $U$,
\[
\nu ( \phi_{t,x_0} (U  ) ) \ge (1-t)^N  \nu (U).
\]
\end{definition}

\begin{remark}
On a $N$-dimensional Riemannian manifold, the measure contraction property $\mathbf{MCP}(0,N)$ is known to be equivalent to non-negative Ricci curvature, see \cite{Ohta}. However, as the next corollary shows, on a $N_1$-dimensional Riemannian manifold, the measure contraction property $\mathbf{MCP}(0,N_2)$ with $N_2 > N_1$ does not imply any Ricci lower bound (such phenomenon was already observed by Rifford \cite{Rifford}).
\end{remark}

As an easy consequence of Theorem \ref{th:MCP}, we deduce:

\begin{corollary}
Let $(\M,\mathcal{F},g)$ be a Sasakian foliation  such that
\[
\mathbf{K}_{\Ho,J} \ge 0,\ \mathbf{Ric}_{\Ho,J^\perp} \ge 0.
\]
Then, for every $\ve > 0$, the metric measure space $(\M,d_\ve, \mu)$ satisfies the measure contraction property $\mathbf{MCP}(0,n+4)$. 
Moreover, the metric measure space $(\M,d_0, \mu)$ satisfies the measure contraction property $\mathbf{MCP}\allowbreak(0,n+3)$ and the constant $n+3$ is sharp.
\end{corollary}

This corollary is interesting because, as observed earlier in Remark \ref{ricci}, the Ricci tensor of the metric $g_\ve$ for the Levi-Civita connection blows up to $-\infty$ in the directions of the horizontal space when $\ve \to 0$. Such similar situations are pointed out in Lee \cite{Lee2}. 

\begin{proof}
Under the assumption, from Theorem \ref{th:SasakianComp3} and Corollary \ref{VL}, we have
\[
\Delta_\Ho r_\ve \le \frac{1}{r_\ve} \min \left\{ 1, \frac{1}{\lambda_\ve} -1\right\} + \frac{n+2}{r_\ve},
\]
and
\[
\Delta_\V r_\ve \le \Xi_{\ve , 0} (r_\ve) = \frac{1}{\ve r_\ve +\frac{r_\ve^3}{12}}.
\]
Therefore,
\[
\Delta_\Ho r_\ve +\ve \Delta_\V r_\ve \le \frac{1}{r_\ve} \min \left\{ 1, \frac{1}{\lambda_\ve} -1\right\} +\frac{n+3}{r_\ve}-\frac{ r_\ve}{12 \ve +r_\ve ^2 }.
\]
As before, we deduce
that for any $A \in \mathcal{B} (\M)$ with $\mu (A) \in ( 0 , \infty )$, 
there exists a dynamic optimal coupling $\Pi$ from $\delta_{x_0}$ to $\bar{\mu}_A$ 
such that the following holds: 
\begin{align} \label{eq:MCP0}
\mu  
& \ge 
( e_t )_\sharp 
\left( t^{n+3+\min \{ 1, \lambda_\ve^{-1} - 1 \} } \sqrt{ \frac{ 12 \ve + \ell (\gamma)^2 }{ 12 \ve +t^2 \ell (\gamma)^2 } }
      \mu (A)\,\Pi \right)
\\ \nonumber 
& \ge  
( e_t )_\sharp \left( t^{n+4} \mu (A)\, \Pi \right). 
\end{align}
Thus the former assertion holds. 
Letting $\ve \to 0$ in \eqref{eq:MCP0} 
yields the latter result.
\end{proof}

\subsection{Horizontal Hessian comparison theorem for the sub-Riemannian distance}

To conclude the paper, we comment on the Hessian comparison theorem in the case that was let open, namely $\ve=0$. It does not seem easy to directly take the limit $\ve \downarrow 0$ in Theorem~\ref{Hessian general3}. However, one can still prove some Hessian comparison theorem for the sub-Riemannian distance with the aid of Theorem~\ref{th:MCP}. For simplicity of the discussion, we restrict ourselves to the case of non-negative horizontal sectional curvature and focus on the worst possible direction in the Hessian comparison theorem.
 
We first prove the following slight improvement of Theorem~\ref{Hessian general3}, in the case $k_1=k_2=0$.

\begin{theorem}\label{Hessian comparison}
Let $(\M,\mathcal{F},g)$ be a Sasakian foliation. Let $\ve >0$. Assume that the horizontal sectional curvature of the Bott connection is non-negative,  namely for all horizontal fields $X,Y$,
\[
\langle R(X,Y)Y, X \rangle_\Ho \ge 0.
\]
 Let $x \neq x_0$ which is not in the $g_\ve$ cut-locus of $x_0$. Let $X \in T_x \M$ which is horizontal and such that $\| X \|_\Ho=1$. 
Then, one has at $x$,
\[
\nabla_\Ho^2 r_\varepsilon (X,X) \le \frac{1}{r_\varepsilon} +\frac{ \langle   X, \nabla_\Ho r_\varepsilon \rangle^2_\Ho}{r_\ve} + \frac{1}{4 \varepsilon} \frac{ r_\varepsilon \| T(X,\nabla_\Ho r_\ve) \|^2_\V}{1+ \frac{\| \nabla_\Ho r_\varepsilon \|^2 r_\varepsilon^2}{12 \varepsilon}}.
\] 
\end{theorem}

\begin{proof}
Let $\gamma$ be the unique length parametrized geodesic connecting $x_0$ to $x$. We consider at $x$ the vertical vector
\[
Z=\frac{1}{2} \frac{T(X,\gamma')}{1+ \frac{r_\varepsilon^2 \| \gamma'\|_\Ho^2}{12 \varepsilon}}.
\]
We still denote by $Z$ the vector field along $\gamma$ which is obtained by parallel transport of $Z$  for the Bott connection $\nabla$.   We will also still denote by $X$ the vector field along $\gamma$ which is obtained by parallel transport of $X \in T_x \M$  for the adjoint connection $\hat{\nabla}^\varepsilon =\nabla+\frac{1}{\varepsilon} J $. We now consider the vector field $Y$ defined along $\gamma$ by:
\[
Y(t)=-\frac{1}{2\varepsilon} t (t-r_\varepsilon) J_Z \gamma' +\frac{t}{r_\varepsilon} X+\left( t -\frac{1}{2\varepsilon}\left( \frac{t^3}{3} -\frac{1}{2} r_\varepsilon t^2\right)\| \gamma'\|_\Ho^2 \right)Z+\frac{t^2}{2r_\varepsilon} T(\gamma',X).
\]
 From Lemma \ref{lemma hessian} and  the index lemma, one has
\[
\nabla_\Ho^2 r_\varepsilon (X,X)  \le  \int_0^{r_\varepsilon} \left( \langle \nabla^\varepsilon_{\gamma'}  Y, \hat{\nabla}^\varepsilon_{\gamma'} Y  \rangle_\varepsilon-\langle \hat{R}^\varepsilon (\gamma',Y)Y ,\gamma'\rangle_\varepsilon \right) dt.
\]
We now observe that
\begin{align*}
\langle \hat{R}^\varepsilon (\gamma',Y)Y ,\gamma'\rangle_\varepsilon &=\langle R(\gamma',Y)Y, \gamma'\rangle_\varepsilon -\| T(Y,\gamma')\|^2_\varepsilon \\
 &=\langle R(\gamma_\Ho',Y_\Ho)Y_\Ho , \gamma_\Ho'\rangle_\Ho -\| T(Y,\gamma')\|^2_\varepsilon \\
 &\ge - \| T(Y,\gamma')\|^2_\varepsilon.
\end{align*}
Therefore we have
\[
\nabla_\Ho^2 r_\varepsilon (X,X)  \le  \int_0^{r_\varepsilon} \left(\langle \nabla^\varepsilon_{\gamma'}  Y, \hat{\nabla}^\varepsilon_{\gamma'} Y  \rangle_\varepsilon +\| T(Y,\gamma')\|^2_\varepsilon \right).
\]
A lengthy but routine computation yields
\begin{align*}
   \int_0^{r_\varepsilon} & \left( \langle \nabla^\varepsilon_{\gamma'}  Y, \hat{\nabla}^\varepsilon_{\gamma'} Y  \rangle_\varepsilon +\| T(Y,\gamma')\|^2_\varepsilon \right) dt \\
  &= \frac{1}{r_\varepsilon}  +\frac{ \langle   X, \nabla_\Ho r_\varepsilon \rangle^2_\Ho}{r_\ve}+\left( \frac{r_\varepsilon}{\varepsilon} +\frac{r_\varepsilon^3}{12\varepsilon^2}  \| \nabla_\Ho r_\varepsilon \|_\Ho^2\right) \| Z \|_\varepsilon^2\\
&\quad+\frac{1}{2\varepsilon} \| T(X,\gamma')\|^2_\varepsilon+ \left(  \frac{r_\varepsilon}{\varepsilon} +\frac{r_\varepsilon^3}{12\varepsilon^2}  \| \nabla_\Ho r_\varepsilon \|_\Ho^2 \right)  \langle Z, T(\gamma',X) \rangle_\varepsilon.
\end{align*}
Using the fact that
\[
Z=\frac{1}{2} \frac{T(X,\gamma')}{1+ \frac{r_\varepsilon^2 \| \gamma'\|_\Ho^2}{12 \varepsilon}},
\]
one gets 
\[
\int_0^{r_\varepsilon} \langle \nabla^\varepsilon_{\gamma'}  Y, \hat{\nabla}^\varepsilon_{\gamma'} Y  \rangle_\varepsilon +\| T(Y,\gamma')\|^2_\varepsilon=\frac{1}{r_\varepsilon}  +\frac{ \langle   X, \nabla_\Ho r_\varepsilon \rangle^2_\Ho}{r_\ve}+\frac{1}{4 \varepsilon} \frac{ r_\varepsilon \| T(X,\nabla_\Ho r_\ve) \|^2_\V}{1+ \frac{\| \nabla_\Ho r_\varepsilon \|^2 r_\varepsilon^2}{12 \varepsilon}}.
\]
The proof is then complete.
\end{proof}

Observe that we always have
\[
 \frac{1}{r_\varepsilon} +\frac{ \langle   X, \nabla_\Ho r_\varepsilon \rangle^2_\Ho}{r_\ve} + \frac{1}{4 \varepsilon} \frac{ r_\varepsilon \| T(X,\nabla_\Ho r_\ve) \|^2_\V}{1+ \frac{\| \nabla_\Ho r_\varepsilon \|^2 r_\varepsilon^2}{12 \varepsilon}} \le \frac{4}{r_\ve}
\]
and therefore
\[
\nabla_\Ho^2 r_\varepsilon (X,X) \le \frac{4}{r_\ve}.
\]

We conclude with the following (non-optimal) sub-Riemannian Hessian comparison theorem.

\begin{theorem}[Sub-Riemannian Hessian comparison theorem]\label{Hessian general2}
Let $(\M,\mathcal{F},g)$ be a Sasak\-ian foliation.   Assume that the horizontal sectional curvature of the Bott connection is non-negative. Let $X \in \Gamma^\infty (\Ho)$ be a smooth vector field such that $\| X \|_\Ho=1$. For   $x \in \M \setminus \mathbf{Cut}_0 (x_0) $, one has 
\[
\nabla_\Ho^2 r_0 (X,X) \le \frac{4}{r_0} .
\]
\end{theorem}

\begin{proof}
A difficulty in the proof is that we have no topological information about the set $\mathbf{Cut}_0 (x_0) \cup_{n \ge 1} \mathbf{Cut}_{1/n} (x_0)$, thus taking pointwise limits is made difficult. It is however possible to bypass this difficulty by using optimal transportation tools.

 Let $x \in \M \setminus \mathbf{Cut}_0 (x_0)$ and $v \in \Ho_x$ with $\|v\|=1$. From Lemma \ref{cutlocus}, we know that $ \M \setminus \mathbf{Cut}_0 (x_0)$ is an open set, so there exists an open set $U$  containing $x_0$ so that $U \subset   \M \setminus \mathbf{Cut}_0 (x_0)$. Then, there exists at least one minimal  sub-Riemannian geodesic $\gamma : [ 0 , 1 ] \to \M$ 
such that $\gamma_{0} = x$ and $\dot{\gamma}_{0} = v $. We can assume that $\gamma$ is included in $U$.

We denote $z := \gamma_1$ and $y := \gamma_{1/2}$. 
Let $( \ve_n )_{n \in \mathbb{N}}$ be 
a decreasing sequence with $\lim_{n \to \infty} \ve_n = 0$. 
For a sufficiently small $\delta > 0$, 
let $A : = B_0 ( z, \delta )$ and apply Theorem~\ref{th:MCP} (2) 
to this choice of $A$. 
Then we can easily see that $( e_t )_\sharp \Pi \le C \mu$ 
for some constant $C = C (t,\delta) > 0$ for each $t$. 
Since $r_{\ve_n}$ is smooth a.e.~for each $n \in \mathbb{N}$, 
Fubini's theorem implies that, 
for $\Pi$-a.e.~sub-Riemannian-minimal geodesics $\sigma$, 
$r_{\ve_n}$ is twice differentiable at $\sigma_t$ 
for each~$n\in\mathbb{N}$.

Then, for each $h \in C^\infty_0 ( ( 0, 1 ) )$ with $h \ge 0$, 
we have 
\begin{align*}
\int_0^1 h'' (t) r_{\ve_n} ( \sigma_t ) \, dt 
& = 
\lim_{\eta \downarrow 0} 
\int_0^1 
\frac{ h (t + \eta) + h (t - \eta) - 2 h (t) }
{\eta^2} r_{\ve_n} ( \sigma_t ) \, dt 
\\
& = 
\lim_{\eta \downarrow 0} 
\int_0^1 
h (t) \frac{ r_{\ve_n} (t + \eta) + r_{\ve_n} (t - \eta) - 2 r_{\ve_n} (t) }
{\eta^2} \, dt 
\\
& = 
\int_0^1 
h (t) \nabla^2 r_{\ve_n} ( \dot{\sigma}_t, \dot{\sigma}_t )
\, dt 
\\
& \le  4
\int_0^1 
\frac{h (t)}{r_{\ve_n} (\sigma_t)} 
\, dt. 
\end{align*}
Now we take $n \to \infty$ in the last inequality, after integration by $\Pi$.  Thus we obtain 
\begin{align*}
\int_{\mathrm{Geo}_0(\M)} 
\int_0^1 h'' (t) r_0 ( \sigma_t ) \, dt 
\Pi ( d \sigma )
& \le  4
\int_{\mathrm{Geo}_0(\M)} 
\int_0^1\frac{h (t)}{r_{0} (\sigma_t)} 
\, dt 
\Pi ( d \sigma ).
\end{align*}
Let 
\begin{align*}
\bar{r} (t) := 4
\int_{\mathrm{Geo}_0(\M)} 
\frac{1}{r_0(\sigma_t)}
\Pi ( d \sigma ). 
\end{align*}
Let $g (t,s) := \min \{ s (1-t) , t (1-s) \}$ be 
the Green function of $- d^2/ds^2$ on $[0,1]$ 
with the Dirichlet boundary condition. 
Then we have 
\[
\int_0^1 h'' (t) \left( \int_0^1 g (t,s) \bar{r} (s) \, d s \right) d t  
= 
- \int_0^1 h (t) \bar{r} (t) \, d t, 
\]
and hence 
\[
\int_0^1 h''(t) 
\left(
    \int_{\mathrm{Geo}_0(\M)} r_0 ( \sigma_t ) \, \Pi ( d \sigma )
    + \int_0^1 g(t,s) \bar{r} (s) \, ds 
\right)
\, d t \le 0 . 
\]
Thus the distributional characterization of convex functions 
(see \cite[Theorem~1.29]{Simon} for instance), yields that 
$$\ds \int_{\mathrm{Geo}_0(\M)} r_0 ( \sigma_t ) \, \Pi ( d \sigma ) 
+ \int_0^1 g(t,s) \bar{r} (s) \, d s$$
is concave since it is continuous in $t$. 
Thus we have
\begin{align*}
 & \frac12 \int_{\mathrm{Geo}_0(\M)} r_0 ( \sigma_0 ) \, \Pi ( d \sigma ) 
+ 
\frac12 \int_{\mathrm{Geo}_0(\M)} r_0 ( \sigma_1 ) \, \Pi ( d \sigma ) 
\\
&\le 
\int_{\mathrm{Geo}_0(\M)} r_0 ( \sigma_{1/2} ) \, \Pi ( d \sigma ) 
+ 
\int_0^1 g( \frac12 , s ) \bar{r} (s) \, d s. 
\end{align*}
Hence, by letting $\delta \downarrow 0$, and using the proof of Theorem~\ref{th:MCP} (2) , we obtain
\begin{align*}
\frac12 r_0 ( x ) 
+ 
\frac12 r_0 ( z ) 
- 
r_0 ( y )
\le 
4 \int_0^1 g( \frac12 , s )\frac{1}{r_0(\sigma_t)}
\, d s. 
\end{align*}
Then the conclusion follows by dividing the last inequality by $d_0 (x,y)^2$ and 
letting $d_0(x,y) = d_0(x,z) \to 0$. 
\end{proof}

\section{Appendix 1: Second variation formulas and index forms}\label{appendix}

In this appendix, for the sake of reference, we collect without proofs several formulas used in the text. The main point is that the classical theory of second variations and Jacobi fields (see \cite{Chavel}) can be reformulated by using a connection which  is not necessarily the Levi-Civita connection. To make the formulas and computations as straightforward and elegant as for the Levi-Civita connection, the only requirement is that we have to work with a metric connection whose adjoint is also metric.

\

Let $(\M,g)$ be a complete Riemannian manifold and $\nabla$ be an affine metric connection on~$\M$. We denote by $\hat{\nabla}$ the adjoint connection of $\nabla$ given by
\[
\hat{\nabla}_X Y=\nabla_X Y -T(X,Y),
\]
where $T$ is the torsion tensor of $\nabla$. We will assume that $\hat{\nabla}$ is a metric connection. This is obviously equivalent to the fact that for every smooth vector fields $X,Y,Z$ on $\M$, one has
\begin{align}\label{skew torsion}
\langle T(X,Y),Z \rangle=-\langle T(X,Z),Y \rangle.
\end{align}
Observe that the connection $(\nabla+\hat{\nabla})/2$ is torsion free and metric, it is therefore the Levi-Civita connection of the metric $g$. Let $\gamma:[0,T] \to \M$ be a smooth path on $\M$. The energy of $\gamma$ is defined as
\[
E (\gamma)=\frac{1}{2}\int_0^T \| \gamma'(t)\|^2 \, dt.
\]

\

Let now $X$ be a smooth vector field on $\gamma$ with vanishing endpoints. One considers the variation of curves $\gamma(s,t)=\exp_{\gamma(t)} ^{\nabla} (s X(\gamma(t)))$ where $\exp^{\nabla}$ is the exponential map of the connection $\nabla$. The first variation of the energy $E(\gamma)$ is given by the formula:
\begin{align*}
\int_0^T \langle \gamma'  , \nabla_\gamma' X + T(X,\gamma') \rangle \, dt
 &=\int_0^T \langle \gamma'  , \hat{\nabla}_\gamma' X  \rangle \, dt
 =-\int_0^T \langle \hat{\nabla}_\gamma' \gamma'  ,  X  \rangle \, dt.
\end{align*}
As a consequence, the critical curves of $E$ are the geodesics of the adjoint connection $\hat{\nabla}$:
\[
\hat{\nabla}_\gamma' \gamma'=0.
\]
These critical curves are also geodesics for $\nabla$ and for the Levi-Civita connection and thus distance minimizing if the endpoints are not in the cut-locus.
One can also compute the second variation of the energy at a geodesic  $\gamma$ and standard computations yield
\begin{align}\label{index general}
\int_0^T \left( \langle \nabla_{\gamma'}  X, \hat{\nabla}_{\gamma'} X  \rangle-\langle \hat{R}(\gamma',X)X,\gamma'\rangle \right) dt
\end{align}
where $\hat R$ is the Riemann curvature tensor of $\hat \nabla$. This is the formula for the second variation with fixed endpoints. This formula  does not depend on the choice of connection $\nabla$. 

\

The index form of a vector field $X$ (with not necessarily vanishing endpoints) along a geodesic $\gamma$ is given by
\begin{align*}
I (\gamma,X,X): & =\int_0^T \left(\langle \nabla_{\gamma'}  X, \hat{\nabla}_{\gamma'} X  \rangle-\langle \hat{R}(\gamma',X)X,\gamma'\rangle \right) dt \\
 &= \int_0^T \left( \langle \nabla_{\gamma'}  X, \hat{\nabla}_{\gamma'} X  \rangle-\langle R(\gamma',X)X,\gamma'\rangle \right) dt.
\end{align*}

\

If $Y$ is a Jacobi field along the geodesic $\gamma$, one has
\[
\hat{\nabla}_{\gamma'} \nabla_{\gamma'} Y =\hat{\nabla}_{\gamma'} \hat{\nabla}_Y \gamma'=\hat{R} (\gamma',Y)\gamma'
\] 
because $\hat{\nabla}_{\gamma'} \gamma'=0$. The Jacobi equation therefore writes
\begin{align}\label{Jacobi}
\hat{\nabla}_{\gamma'} \nabla_{\gamma'} Y=\hat{R} (\gamma',Y)\gamma'.
\end{align}

We have then the following results:
\begin{lemma}\label{lemma hessian}
Let $x_0 \in \M$ and $x \neq x_0$ which is not in the cut-locus of $x$. We denote by $r=d(x_0,\cdot)$ the distance function from $x_0$. Let $X \in T_x \M$ be orthogonal to $\nabla r (x)$. At the point $x$, we have
\begin{align*}
\nabla^2 r (X,X)=I (\gamma,Y,Y)
\end{align*}
where $\gamma$ is the unique length parametrized geodesic connecting $x_0$ to $x$ and $Y$ the Jacobi field along $\gamma$ such that $Y(0)=0$ and $Y(r(x))=X$.
\end{lemma}

Combining this with the index lemma  yields:

\begin{lemma}\label{index lemma}
Let $x_0 \in \M$ and $x \neq x_0$ which is not in the cut-locus of $x$. Let $X \in T_x \M$. At $x$, we have
\begin{align*}
\nabla^2 r (X,X) \le  \int_0^{r} \left(\langle \nabla_{\gamma'}  \tilde X, \hat{\nabla}_{\gamma'} \tilde X  \rangle-\langle \hat{R}(\gamma',\tilde X) \tilde X,\gamma'  \rangle \right) dt
\end{align*}
where $\gamma$ is the unique length parametrized geodesic connecting $x_0$ to $x$ and  $\tilde X$ is any vector field along $\gamma$ such that $\tilde X(0) =0$ and $\tilde X(r(x))=X$.
\end{lemma}

\section{Appendix 2: Jacobi fields on Sasakian manifolds of constant sectional curvature}

In this Appendix, we compute the Jacobi fields of the metric $g_\ve$ on Sasakian manifolds of constant sectional curvature. An important difference with respect to \cite{BD} is that we work with the adjoint connection $\hat{\nabla}^\ve=\nabla+\frac{1}{\varepsilon} J$ instead of the Tanaka-Webster (= Bott) connection.
\

In this appendix, we assume that the Riemannian foliation on $\M$ is a Sasakian foliation. As in Section 3,  the Reeb vector field on $\M$ will be denoted by $S$ and the complex structure by $\mathbf{J}$. We refer for instance to Chapter 2 in \cite{Wang} for a discussion about Sasakian model spaces from the point of view of sub-Riemannian geometry.

\

We use the notations of Section~\ref{sec:HTypeComp}. For any vector field $Y$ along $\gamma$, we will use $Y'$ for the covariant derivative with respect to $\hat{\nabla}^\ve_{\gamma'}$. Whenever we use the word parallel, it is with respect to $\hat \nabla^\ve$. We identify vectors and their corresponding parallel vector field. We define $\phi_\mu$, $\psi_\mu$ and $\Phi_\mu$ as in Section~\ref{sec:HTypeComp}
\begin{lemma} \label{lemma:ExJacobi}
Let $R$ be the curvature of the Bott connection $\nabla$ and assume that for some $k \in \mathbb{R}$,
$$\langle R(v,w) w, v \rangle = k \|v \wedge w \|^2_{g_\Ho}, \qquad v,w \in \Ho.$$
Let $\gamma:[0,r] \to \M$ be a geodesic of unit speed with respect to $g_{\ve}$.
\begin{enumerate}[\rm (a)] 
\item Let $Y$ be the Jacobi vector field along $\gamma$ such that $Y(0) = 0$ and $Y(r) = v_0 \in \Ho_{\gamma(r)}$. Assume that $\gamma_\Ho' \neq 0$ and that $v_0$ is orthogonal to $\mathbf J\gamma'_\Ho$ and $\gamma'_\Ho$. Finally, if
$$\mu_\gamma = \frac{\| \gamma_{\Ho}'\|^2_g - 1}{4\varepsilon} - \| \gamma_\Ho'\|^2_g k <0,$$
assume that $\sqrt{-\mu_\gamma} r < \pi$. Then
\begin{equation} \label{Jac1} Y(t) =\frac{\phi_{\mu_\gamma}(t)}{\phi_{\mu_\gamma}(r)} \left(\cos \left( \frac{\langle S, \gamma' \rangle_{g_\ve} (r-t)}{2} \right) v_0- \sin \left( \frac{\langle S, \gamma' \rangle_{g_\ve} (r-t)}{2}\right)\mathbf{J} v_0\right).\end{equation}
\item Assume that $\gamma$ is a horizontal curve. Let $Y$ be the Jacobi vector field along $\gamma$ such that $Y(0) = 0$ and $Y(r) = \mathbf{J} \gamma_\Ho'(r) $. If $k >0$, assume that $\sqrt{k} r \leq \pi$. Then
\begin{equation} \label{Cve} C_\ve = \psi'_{-k}(r)^2 - \psi_{-k}(r) \psi_{-k}''(r) + \ve r \psi''_{-k}(r) >0,\end{equation} and $Y(t)$ is given by
\begin{align*} Y(t) &= \frac{1}{C_\ve} \Big( \psi_{-k}'(r) \psi_{-k}'(t) + (\ve r- \psi_{-k} (r)) \psi_{-k}''(t)   \Big)\mathbf{J} \gamma'(t) \\ 
& \quad + \frac{1}{C_\ve}\Big( \psi_{-k}'(r) \psi_{-k}(t) - \psi_{-k}(r) \psi_{-k}'(t) + \ve(r\psi_{-k}'(t) - t\psi_{-k}'(r))  \Big)S(t). \end{align*}
\item Let $\gamma: [0,r] \to \mathbb{M}$ be vertical with $r < 2 \pi \sqrt{\ve}$. Then any Jacobi field $Y$ with $Y(0) = 0$ and $Y(r) =v_0 \in \Ho_x$, is given by
\begin{align}  \nonumber
Y(t) &= \frac{1}{2 \left(1- \cos \frac{r}{\sqrt{\ve}} \right)} \left( \left( 1+  \cos \frac{r-t}{\sqrt{\ve}} - \cos \frac{r}{\sqrt{\ve}} - \cos \frac{t}{ \sqrt{\ve} }  \right) v_0 \right. \\ \label{Jac2}
& \left.  \qquad \qquad \qquad \qquad - \langle \sqrt{\ve}  S, \gamma' \rangle_{g_\ve} \left( \sin \frac{r-t}{\sqrt{\ve}} - \sin\frac{r}{\sqrt{\ve}} + \sin \frac{t}{ \sqrt{\ve} } \right) \mathbf{J} v_0 \right) 
\end{align}
In fact, this is a Jacobi-field along all vertical geodesics on a Sasakian manifold without any assumption on the curvature.
\end{enumerate}
\end{lemma}

\begin{proof}
 The Jacobi equation for a vector field $Y$ is given by
\begin{align} \nonumber
0 & = \hat{\nabla}^\ve_{\gamma'} \nabla_{\gamma'}^\ve Y - \hat{R}^\ve(\gamma', Y) \gamma' \\ \label{RealJacobi}
& =  Y'' - T( \gamma', Y') - \frac{1}{\ve} J_{\gamma'} Y'+ \frac{1}{\ve} J_{Y'}  \gamma'  - R(\gamma', Y) \gamma' - \frac{1}{\ve} J_{T( \gamma',Y)}  \gamma'. 
\end{align}
Define $X = \pi_\Ho Y $ and $\langle Y, S\rangle_g  = F$. Equation \eqref{RealJacobi} then becomes
\begin{eqnarray*}
F'' &=& \langle \mathbf{J} \gamma'_\Ho, X' \rangle_g , \\
X'' &=& -k \| \gamma'_\Ho\|^2_g X  +  \langle S, \gamma' \rangle_{g_\ve} \mathbf{J} X' - \frac{1}{\ve} \left( F' - \langle \mathbf{J} \gamma'_\Ho, X\rangle_g \right) \mathbf{J} \gamma'_\Ho.
\end{eqnarray*}
Note that $C_0 = F' - \langle \mathbf{J} \gamma'_\Ho, X\rangle_g$ is constant and so $C_0 \mathbf{J} \gamma'_\Ho$ is a parallel vector field.
\begin{enumerate}[\rm (a)]
\item We assume that $X$ is contained in the orthogonal complement of $\mathbf{J}\gamma'_\Ho$ and $\gamma'_\Ho$ which is a parallel vector bundle along $\gamma$. Given this assumption and initial condition $F(0)=0$, we have $F = C_0 t$. From the condition $F(r) = 0$, we must have $C_0 = 0$, and so $X$ is a solution of
$$X'' = - k \| \gamma_\Ho'\|^2_g X + \langle S, \gamma' \rangle_{g_\ve} \mathbf{J} X'.$$
Define $Z = X + i\mathbf{J} X$. Then
$$Z'' + i \langle S, \gamma' \rangle_{g_\ve} Z' + \| \gamma_\Ho' \|^2_g k Z=0 .$$
The solution with initial condition $Z(0) = 0$ is
$$Z(t) = e^{-i \langle S, \gamma' \rangle_{g_\ve} t/2} \phi_{\mu_\gamma}(t) (X_0 + i \mathbf{J} X_0),$$
where $X_0$ is some parallel vector field and $$\mu_\gamma = - \frac{\langle S, \gamma' \rangle^2_{g_\ve}}{4} -  k  \| \gamma_\Ho' \|^2_g =  - \frac{1- \| \gamma'_\Ho\|^2_g}{4\varepsilon} -  k  \| \gamma_\Ho' \|^2_g .$$ Using that $Z(r) = v_0 + i \mathbf{J} v_0$
and taking the real part of $Z$, we get the result.
\item If $\gamma$ is horizontal, then we are left so solve
\begin{eqnarray*}
F'' &=& \langle \mathbf{J} \gamma', X' \rangle_g , \\
X'' &=& -k X - \frac{1}{\ve} C_0 \mathbf{J} \gamma'.
\end{eqnarray*}
Write $X = f \mathbf{J} \gamma'$. Then $C_0 = F' - f$ and
$$f'' + kf =-  \frac{1}{\ve} C_0 .$$
The solution, given the initial condition, is
$$f(t) = C_1\phi_{-k}(t) - \frac{C_0}{\ve} \int_0^t \phi_{-k}(s) ds = C_1 \psi_{-k}''(t) - \frac{C_0}{\ve} \psi_{-k}'(t),$$
for some constant $C_1$. This means that
$$F(t) = C_0\left( t - \frac{1}{\ve} \psi_{- k}(t)  \right) + C_1 \psi_{-k}'(t).$$
Then we need to solve the equations $C_1 \psi_{-k}''(r) - C_0 \frac{1}{\ve} \psi_{-k}'(r) =1$ and $C_0 ( r - \frac{1}{\ve} \psi_{-k}(r)) + C_1 \psi_{-k}' (r) = 0$. If $C_\ve$ is as in \eqref{Cve}, the solution is
$$C_0 =-  \frac{\ve \psi_{-k}'(r)}{C_\ve}, \qquad C_1 = \frac{\ve r - \psi_{-k}(r)}{C_\ve} \, .$$
To complete the proof, we need to show that the denominator is in fact non-zero. However, this follows from the observation that
$$\psi'_{\mu}(r)^2 - \psi_{\mu}(r)'' \psi_\mu(r) = r \psi_\mu''(r) \Psi_\mu(r/2),$$
so $C_\ve = r \psi''_\mu(r) (\Psi_\mu(r/2) + \ve) = r \phi_\mu(r) (\Psi_\mu(r/2) + \ve)$. 
\item Define $s =\langle \sqrt{\ve}  S, \gamma' \rangle_{g_\ve} \in \{ -1, 1\}$. Since $R(\gamma', \cdot) = 0$, when $\gamma$ is vertical, we need no assumptions on the curvature. The equation for a Jacobi vector field is now
\begin{eqnarray*}
F'' &=& 0 , \\
X'' &=&  \langle S, \gamma' \rangle_{g_\ve} \mathbf{J} X' = \frac{s}{\sqrt{\ve}} \mathbf{J} X'.
\end{eqnarray*}
Define $Z = X + i \mathbf{J} X$. With initial and final conditions, we have $F= 0$ and
$$Z(t) = \frac{ 1- e^{-i st/ \sqrt{\ve} } }{1- e^{-i s r /\sqrt{\ve} } } (v_0 + i \mathbf{J} v_0).$$
Taking the real part, the result follows.\qedhere
\end{enumerate}
\end{proof}

\end{document}